\documentclass[a4paper,12pt,oneside]{article}
\usepackage{graphicx}

\usepackage[centertags]{amsmath}
\usepackage{graphicx}
\usepackage{amsfonts}
\usepackage{amssymb}
\usepackage[all]{xy}
\usepackage{color}
\usepackage{enumerate}

\usepackage[utf8]{inputenc}
\usepackage{tikz-cd}
\usepackage{hyperref}
\usepackage{mathtools}
\usepackage{import}

\usepackage{relsize}
\usepackage[all]{xy}
\usepackage{microtype}
\usepackage{xspace}
\usepackage{soul}
\usepackage{xcolor}

\usepackage{geometry}

\def\squarebox#1{\hbox to #1{\hfill\vbox to #1{\vfill}}}
\newcommand{\qed}{\hspace*{\fill}
\vbox{\hrule\hbox{\vrule\squarebox{.667em}\vrule}\hrule}\smallskip}

\newtheorem{theorem}{Theorem}[section]
\newtheorem{lemma}[theorem]{Lemma}

\newtheorem{proposition}[theorem]{Proposition}
\newtheorem{definition}[theorem]{Definition}

\newenvironment{proof}{\noindent {\bf Proof:}}{\hfill $\qed $ \newline}

\newcommand{\R}{{\mathbb R}}

\newcommand{\Z}{{\mathbb Z}}
\newcommand{\C}{\mathbb{C}}
\newcommand{\F}{\mathbb{F}}
\newcommand{\B}{{\mathbb B}}

\newcommand{\simto}{\stackrel{\sim}{\longrightarrow}}

\renewcommand{\S}{{\mathcal S}}
\renewcommand{\B}{{\mathcal B}}
\DeclareMathOperator{\Ima}{Im}

\newcommand{\ad}{{\rm ad}}
\newcommand{\Ad}{{\rm Ad}}

\newcommand{\Sl}{{\rm Sl}}
\newcommand{\SO}{{\rm SO}}
\newcommand{\Sp}{{\rm Sp}}

\newcommand{\e}{{\rm e}}

\renewcommand{\min}{\mbox{{\rm min}}}
\newcommand{\SU}{{\rm SU}}

\newcommand{\g}{\mathfrak{g}}
\renewcommand{\k}{\mathfrak{k}}
\newcommand{\s}{\mathfrak{s}}
\renewcommand{\a}{\mathfrak{a}}

\newcommand{\m}{\mathfrak{m}}

\newcommand{\n}{\mathfrak{n}}

\renewcommand{\sl}{\mathfrak{sl}}

\newcommand{\q}{\mathfrak{q}}

\begin{document}
\thispagestyle{empty}

\title{Cellular homology of compact groups: \\  Split real forms}

\author{Mauro Patrão \and Ricardo Sandoval}

\maketitle

\begin{abstract}
In this article, we use the Bruhat and Schubert cells to calculate the cellular homology of the maximal compact subgroup $K$ of a connected semisimple Lie group $G$ whose Lie algebra is a split real form. We lift to the maximal compact subgroup the attaching maps for the maximal flag manifold presented in \cite{lonardo} and use it to characterize algebraically the incidence order between Schubert cells. We also present algebraic formulas to compute the boundary maps which extend to the maximal compact subgroups similar formulas obtained in the case of the maximal flag manifolds. Finally, we apply our results to calculate the cellular homology of $\SO(3)$ as the maximal compact subgroup of $\Sl(3, \R)$ and the cellular homology of $\SO(4)$ as the maximal compact subgroup of the split real form $G_2$.
\end{abstract}


\section{Introduction}

The construction of cellular decompositions of group manifolds and homogeneous spaces is an old theme. For the classical compact Lie groups one can use cells using products of reflections via the method that goes back to Whitehead \cite{whitehead} and was later developed by \cite{yokota1}, \cite{yokota2}. More recently it can also be found in section 3.D of \cite{hatcher}. The degeneration of Spectral sequences that occurs for unitary and sympletic groups fails for the orthogonal groups, because in the analogue of the iterated fiber decomposition of the orthogonal groups one encounters spheres of adjacent dimensions (see section 3.2 of \cite{carlson}).

In \cite{lonardo}, the Bruhat and Schubert cells are used in flag manifolds to divide them into cell complexes to then calculate the cellular homology of generalized real flag manifolds. Formulas for the boundary operator are then found.
The results of \cite{lonardo} were already partially found for flag manifolds by Kocherlakota \cite{kocher}. In the realm of Morse homology, it is proven in Theorem 1.1.4 of \cite{kocher}  that the boundary operator for the Morse-Witten complexes are intimately related to the Bruhat cells, since they are the unstable manifolds of the gradient flow of a Morse function (see Duistermatt-Kolk-Varadarajan \cite{dkv}). Nevertheless the cellular point of view of \cite{lonardo} has the advantage of showing the geometry in a more evident way, in particular, the choice of minimal decompositions for the elements of the Weyl group $W$ fix certain signs which are ambiguous in the Morse-Witten complex. For each element $w \in W$, we have a Bruhat cell $\B(w)$ in the maximal flag manifold $\F$ of a connected semisimple Lie group $G$, and its closure, called Schubert cell, is given by
\[
\S(w) = \coprod_{w'\leq w} \B(w')
\]
where $\leq$ is the so-called Bruhat-Chevalley order.

In this article, we use the Bruhat and Schubert cells to calculate the cellular homology of the maximal compact subgroup $K$ of $G$ whose Lie algebra is a split real form. We follow \cite{lonardo} to first construct the skeleton and then to find algebraic expressions for the degrees of the maps. We apply these results to calculate the cellular homology of $\SO(3)$ as the maximal compact subgroup of $\Sl(3, \R)$ and calculate the homology of $\SO(4)$ as the maximal compact subgroup of the split real form $G_2$ (see \cite{yokota}). Denote by $U$ and $C$ the groups of the connected components of, respectively, the normalizer and the centralizer in $K$ of a maximal abelian subalgebra in the symmetric component of the Cartan decomposition. If $\pi: K \to \F$ is the natural projection, then $\pi(U) = U/C = W$ and, for each simple reflection $r_\alpha \in W$, we can choose a suitable element $s_\alpha \in U$ such that $\pi(s_\alpha) = r_\alpha$ and $s_\alpha^2 = c_\alpha \in C$. Since the simple reflections generate $W$, each element $u \in U$ can be written as $u = s_1 \cdots s_dc$, where $c \in C$ and $\pi(u) = r_1 \cdots r_d$ is a reduced expression. In this case, the length of $\pi(u)$ is given by $\ell(\pi(u)) = d$. For each element $u \in U$, we have a Bruhat cell $\B(u)$ in $K$, and its closure, called Schubert cell, is given by
\[
\S(u) = \coprod_{u'\leq u} \B(u')
\]
where $\leq$ is an order in $U$ characterized by our first main theorem.

\begin{theorem}\label{theo-1}
Let $u, u' \in U$. If $\ell(\pi(u)) = d$ and $\ell(\pi(u')) = d-1$, then $u' < u$ if and only if for some (or equivalently for each) reduced expression $\pi(u) = r_1 \cdots r_d$ we have that $u = s_1 \cdots s_dc$, for some $c \in C$, and that
\[
u' = u_i^0 = s_1 \cdots s_{i-1}s_{i+1} \cdots s_d c
\qquad
\mbox{or}
\qquad
u' = u_i^1 = s_1 \cdots s_{i-1}c_is_{i+1} \cdots s_d c
\]
for some $1 \leq i \leq d$, where $c_i = s_i^2$. Furthermore, $u' < u$ if and only if there exists $u_1,\dots,u_n$ such that $u_1 = u'$, $u_n = u$, $\ell(\pi(u_{k+1})) = \ell(\pi(u_k)) + 1$ and $u_k < u_{k+1}$.
\end{theorem}

In the case of $\SO(3)$ as the maximal compact subgroup of $\Sl(3, \R)$, we have that
\begin{eqnarray*}
U
& = &
\{\textcolor{blue}{1, c_1, c_2, c_1c_2}\}
\cup \\
&   &
\cup \{\textcolor{red}{s_1, s_1c_1, s_1c_2, s_1c_1c_2, s_2, s_2c_1, s_2c_2, s_2c_1c_2}\} \cup \\
&   &
\cup \{\textcolor{green}{s_2s_1, s_2s_1c_1, s_2s_1c_2, s_2s_1c_1c_2, s_1s_2, s_1s_2c_1, s_1s_2c_2, s_1s_2c_1c_2}\}
\cup \\
&   &
\cup \{\textcolor{orange}{s_1s_2s_1, s_1s_2s_1c_1, s_1s_2s_1c_2, s_1s_2s_1c_1c_2}\}
\end{eqnarray*}
and the order $<$ is given by the graph bellow, where $u' < u$ if there is a sequence of arrows starting from $u'$ and ending in $u$. The labels of the red and green vertices where suppressed in order to get a better visualization of the arrows.

\begin{center}
\includegraphics[scale=0.6]{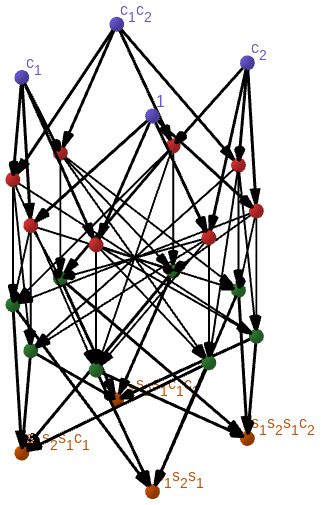}
\end{center}

This algebraic order is also related to the dynamical order between Morse components of the dynamics considered in \cite{Patrao3} and to the dynamical order between control sets of the semigroup actions considered in \cite{Patrao2}. In our second main theorem, the boundary maps are characterized algebraically.

\newpage

\begin{theorem}\label{theo-2}
Let $\mathcal{C}$ be the $\Z$-module freely generated by $\B (u)$, $u \in U$. The  boundary map $\delta : \mathcal{C} \to \mathcal{C}$ is such that, if $u = s_1 \cdots s_dc$ and $\pi(u) = r_1 \cdots r_d$ is a reduced expression, then
\[
\delta \B (u) = \sum_{v < u} \rho (u, v) \B (v)
\]
with $\rho (u, v) = 0$, except when
\[
v = u_i^0 = s_1 \cdots s_{i-1}s_{i+1} \cdots s_d c
\qquad
\mbox{or}
\qquad
v = u_i^1 = s_1 \cdots s_{i-1}c_is_{i+1} \cdots s_d c
\]
and $\pi(v) = r_1 \cdots r_{i-1}r_{i+1} \cdots r_d$ is a reduced expression, and thus
\[
\rho(u, u_i^0) = \varepsilon(u)_i (-1)^i
\]
and
\[
\rho(u, u_i^1) = \varepsilon(u)_i (-1)^{i + 1 + \sigma}
\]
where $\varepsilon(u)_i = \pm 1$ and
\[
\sigma = \sum_{\beta \in \Pi_i} \frac{2\langle \alpha_i, \beta \rangle}{\langle \alpha_i, \alpha_i \rangle}
\]
with $\Pi_i = \Pi^+ \cap r_{i+1} \cdots r_d \Pi^-$, and $\Pi^\pm$ are the sets of positive and negative roots.
\end{theorem}

In the case of $\SO(3)$ as the maximal compact subgroup of $\Sl(3, \R)$ and in the case of $\SO(4)$ as the maximal compact subgroup of the split real form $G_2$, we can always choose $\varepsilon(u)_i = 1$ in the above formulas. The paper has the following structure. In Section 2, we set the notation and recall the necessary definitions and results about homogeneous spaces of
Lie groups and semi-simple Lie theory. In Section 3, we introduce the Bruhat and Schubert cells in our context and derive some elementary results. In Section 4, we lift to the maximal compact subgroup the attaching maps of the maximal flag manifold and use it to prove Theorem \ref{theo-1}, characterizing the above order. In Section 5, we construct the skeleton and the boundary map, while, in the Section 6, we prove Theorem \ref{theo-2}. Finally, in Section 7, we apply our main results to compute the cellular homology of the two examples mentioned before. Some calculations are left to the Appendix.

\section{Preliminaries}\label{secpreliminar}

\subsection{Homogeneous spaces of Lie groups}
\label{homogspaces}

For the theory of Lie groups and its homogeneous spaces we refer to Hilgert and Neeb \cite{neeb} and Knapp \cite{knapp}.  Let $G$ be a real Lie group with Lie algebra $\g$ where $g \in G$ acts on $X \in \g$ by the adjoint action
$g X = \Ad( g ) X$.  We have that  $\Ad( \exp(X) ) = e^{\ad(X)}$
where $\exp: \g \to G$ is the exponential of $G$, $\Ad$ and $\ad$ are, respectively, the adjoint representation of $G$ and $\g$. Let a Lie group $G$ act on a manifold $F$ on the left by the differentiable map
$
G \times F \to F$,  $(g,x) \mapsto gx
$.
Fix a point $x \in F$.  The isotropy subgroup $G_{x}$ is the set of all $g \in G$ such that $gx = x$.  We say that the action is transitive or, equivalently, that $F$ is a homogeneous space of $G$, if $F$ equals the orbit $Gx$ of $x$ (and hence the orbit of every point of $F$).
In this case, the map
$$
G \to F\qquad g \mapsto g x
$$
is a submersion onto $F$ which is a differentiable locally trivial principal fiber bundle with structure group the isotropy subgroup $G_{x}$.
Quotienting by $G_{x}$ we get the diffeomorphism
$$
G/G_{x} \simto F\qquad g G_{x} \mapsto g x
$$
Since the map $G \to F$ is a submersion, the derivative of the map $g \mapsto gx$ on the identity gives the infinitesimal action of $\g$, more precisely, a surjective linear map
$$
\g \to TF_x \qquad Y \mapsto Y \cdot x
$$
whose kernel is the isotropy subalgebra $\g_x$, the Lie algebra of $G_x$.
The derivative of the map $g: F \to F$, $x \mapsto gx$, gives the action of
$G$ on tangent vectors $g v = D g (v)$, $v \in TF$, which is related to the infinitesimal action by
$$
g( Y \cdot x ) = gY \cdot gx
$$
For a subset $ \q \subset \g$, denote by $\q \cdot x$ the set of all tangent vectors $Y \cdot x$, $Y \in \q$.
It follows that $ TF_{gx} = g( \g \cdot x ) $.

\subsection{Semi-simple Lie theory}\label{section-lie}

For the theory of real semisimple Lie groups and their flag manifolds we refer to Duistermat-Kolk-Varadarajan \cite{dkv},
Helgason \cite{helgason}, Hilgert and Neeb \cite{neeb}, and Knapp \cite{knapp}.
Let $G$ be a connected real Lie group with finite center such that its Lie algebra $\g$ is a split real form. Fix a Cartan decomposition $\g = \k \oplus \s$ and denote by $\theta$ the Cartan involution and by $\langle \cdot, \cdot \rangle$ the associated Cartan inner product.
Let $K$ be the connected subgroup with Lie algebra $\k$, it is a maximal compact subgroup of $G$.
Since $\ad(X)$ is anti-symmetric for $X \in \k$, the Cartan inner product is $K$-invariant.
Since $\ad(X)$ is symmetric for $X \in \s$, a maximal abelian subspace $\frak{a} \subset \frak{s}$ can be simultaneously diagonalized so that $\g$ splits as an orthogonal sum of
$$
\g_\alpha = \{ X \in \g:\, \ad(H)X = \alpha(H)X, \, \forall H \in \a \}
$$
where $\alpha \in \a^*$ (the dual of $\a$). We have that $\g_0 = \m \oplus \a$, where $\m$ is the centralizer of $\a$ in $\k$. Since $\g$ is a split real form, each $\g_{\alpha}$ is one dimensional and that $\m$ is trivial. A root is a functional $\alpha \neq 0$ such that its root space $\g_\alpha \neq 0$, denote the set of roots by $\Pi$. We thus have the root space decomposition of $\g$, given by the orthogonal sum
$$
\g = \m \oplus \a \oplus \sum_{\alpha \in \Pi} \g_\alpha
$$
Fix a Weyl chamber $\frak{a}^{+}\subset \frak{a}$ and let $\Pi^{+}$ be the corresponding positive roots, $\Pi^- = - \Pi^+$ the negative roots and $\Sigma $ the set of simple roots.  Consider the nilpotent subalgebras
\[
\n^\pm = \sum_{\alpha \in \Pi^{\pm}}\frak{g}_{\alpha }
\]
such that
\[
\g = \m \oplus \a \oplus \n \oplus \n^-
\]

In this paper, we look at $K$ as the homogeneous manifold $G/AN$, denoting its base, the left coset $AN$, as $b$. The natural action of $G$ in $G/AN$ is given by left multiplication (as in section 10.1 of \cite{neeb}). From the Iwasawa decomposition, it follows that the map $K \to G/AN$ given by $k \mapsto k b$, is a $K$-equivariant diffeomorphism. The isotropy subalgebra of the base $b$ is given by $\a \oplus \n$, while the isotropy subalgebra of $k b$ for $k \in K$ is given by
\[
\g_{k b} = k (\a \oplus \n)
\]
We also look at the maximal flag manifold $\F$ of $G$ as the homogeneous manifold $G/MAN$, denoting its base, the left coset $MAN$, as $b_0$. and is diffeomorphic to $K/M$. The natural projection $K \to \F$ given by $kb \mapsto k b_0$, is a $M$-principal bundle with the action of the structural group given $(kb,c) \mapsto kmb$. This projection is equivalent to the natural projection $K \to \F$ given by $k \mapsto k M$ and both will be denoted by $\pi$. Note that this map is a differentiable finite covering of $\F$ with $|M|$ sheets. Then for a sufficiently small open set $V \subset \F$ there is a diffeomorphism $\pi^{-1}: V \to B$, where $B$ is an open set of $K$, such that $\pi \pi^{-1} = \textnormal{id} |_V$ and $\pi^{-1} \pi|_B =  \textnormal{id} |_B$. Note also that, for any $m \in M$ and $x \in V$, $\pi (\pi^{-1}(x)m) = x$, so that, for a given $V$, $\pi^{-1}$ is not uniquely defined and in fact has $|M|$ possible choices.

The Weyl group $W$ is the finite group generated by the reflections over the root hyperplanes $\alpha=0$ in $\frak{a}$, $\alpha \in \Pi$. $W$ acts on $\frak{a}$ by isometries and can be alternatively be given as $W=M^{*}/M$ where $M^{*}$ and $M$ are the normalizer and the centralizer of $\a$ in $K$,
respectively. An element $w$ of the Weyl group $W$ can act in $\g$ by taking a representative in $M^*$.  This action centralizes $\a$, normalizes $\m$, permutes the roots $\Pi$ and thus permutes the root spaces $\g_\alpha$, where $w \g_\alpha = \g_{w \alpha}$ does not depend on the representative chosen in $M^*$. Since $\g$ is a split real form and $\Ad(G)$ can be complexified, following the Theorems 7.53 and 7.55 from \cite{knapp}, we get that $\Ad(M) = F$ where $F$ is the cartesian product of cyclic groups of order 2 and is generated by
\[
\gamma_\alpha  = \exp 2 \pi i |\alpha|^{-2}H_{\alpha}
\]
Considering each element $ w \in W$ as a product of simple reflections $r_{\alpha}$, the {\em length} of $\ell (w)$ of $w \in W$ is the number of simple reflections in any reduced expression of $w$. Another useful result Theorem 4.15.10 of \cite{vara} is that $l (w)$ is equal to the cardinality of $\Pi_w = \Pi^+ \cap w \Pi^-$, or the set of positive roots sent to the negative roots by $w^{-1}$. Let $w = r_1 \cdots r_d$ be a fixed reduced expression of $w$ and $\alpha_i = \alpha_{r_i}$ be the simple positive roots for each $r_i$, then each root of $\Pi_w$ can be written explicitly as in Theorem 4.15.10 of \cite{vara} as
\begin{equation}
\label{vara}
\Pi_w = \{ \alpha_1, r_1 \alpha_2 , \dots, r_1 \cdots r_{d-1} \alpha_d \}
\end{equation}

\section{Bruhat and Schubert cells}\label{sec-bruhat-schubert}

We start introducing some elementar definitions and recalling some dynamical results both on $K$ and on $\F$. Since $\m$ is trivial, we have that $M_0 = 1$, that
\[
C=M/M_0=M
\qquad
\mbox{and}
\qquad
U=M_*/M_0=M_*
\]
and that
\[
\pi(U) = M^*/M = W
\]

\begin{definition}
Each $u \in U$ define a Bruhat cell in $K$ given by
\[
\B(u):= N u b
\]
while each $w \in W$ define a Bruhat cell in $\F$ given by
\[
\B(w):= N w b_0
\]
and we have that
\[
\pi(\B(u)) = \B(\pi(u))
\]
\end{definition}

By Theorem 3.9 of \cite{Patrao3}, choosing a regular element $H$, we get the Bruhat decomposition of $K$, given by
\[
K = \coprod N u b = \coprod \B (u)
\]
where the disjoint unions are taken over all $u \in U$, and, by Proposition 1.3 of \cite{dkv}, we get the Bruhat decomposition of $\F$, given by
\[
\F = \coprod N w b_0 = \coprod \B (w)
\]
where the disjoint unions are taken over all $w \in W$. Note that if one prefers to use instead the stable manifolds decomposition its possible to obtain similar results by using that $N^- = u^- N (u^-)^{-1}$ where $u^- \in w^- M$ and $w^-$ is the principal involution.

In an analogous fashion to flags we define in $K$, the Schubert cells to be the closure of the Bruhat cells.

\begin{definition}
Each $u \in U$ define a Schubert cell in $K$ given by
\[
\S(u):= \textnormal{cl} (\B (u))
\]
while each $w \in W$ define a Schubert cell in $\F$ given by
\[
\S(w):= \textnormal{cl} (\B (w))
\]
and we have that
\[
\pi(\S(u)) = \S(\pi(u))
\]
\end{definition}

Let $\textnormal{cl} A$ be the closure of the set $A$. In the present work we take $\partial (A)$ to mean the {\em frontier} of the set $A$, that is $\textnormal{cl} A \backslash A$. Note that in this case the frontier of the Bruhat cell is then $\partial \B (u) = \S(u) \backslash \B (u)$ and also $\partial \B (w) = \S(w) \backslash \B (w)$. Later, to avoid confusion of notation, we use $\delta$ to be the boundary operator.

\begin{lemma} \label{border of bruhat}
The frontier of a Bruhat cell, $\partial \B (u) := \S (u) \backslash \B (u) $ in $K$ is the union of Bruhat cells $\B (v) = N v b$ for $v \in V (u)$, where $V (u)$ is a subset of $U$. The cells in $\partial \B (u)$ have smaller dimension then $\B(u)$.
\end{lemma}

\begin{proof}
Let $n v b$ be in the frontier of $\B (u) = N u b$ for some $n \in N$ then $\B (v) = N v b$ is also in the frontier of $N u b$. In fact, since $n v b$ is in the frontier of $N u b$ then there is a sequence $n_k$ in $N$ such that $n_k u b \to n v b$. Let $y$ be any point in $N v b$ then $y = n' v b$ for some $n' \in N$ and the sequence in $N u b$, $n' n^{-1} n_k u b$, converges to $n' n^{-1} n v b = n' v b = y$ and $y$ is also in the frontier of $\B (u)$.

Since $\pi$ is a finite cover it preserves dimension of cells. Let $\pi(u)=w$ then $\pi(\partial \B(u)) \subset \partial \B(w)$, now from Proposition 1.9 (2) of \cite{lonardo} the dimension of $\partial \B(w)$ is less then the dimension of $\B(w)$.
\end{proof}

So for each $u \in U$ there is {\em a subset} $V (u)$ of $U$ so that
\[
\partial \B(u) = \bigcup_{v \in V(u)} \B(v)
\]
so that,
\[
\S(u) = \B(u) \cup \bigcup_{v \in V(u)} \B(v)
\]
and since multiplication by right by an element $c \in C$ takes Bruhat cells to Bruhat cells
\[
\S(uc) = \S(u)c,
\qquad
\B(uc) = \B(u)c,
\qquad
\mbox{and}
\qquad
\partial \B(uc) = \partial \B(u)c
\]
for any $c \in C$, so that $V(uc)=V(u)c$. In the maximal flag manifold, we have that
\[
\partial \B(w) = \bigcup_{w' < w} \B(w')
\]
where the Bruhat-Chevalley order in $W$ is such that $w' < w$ if and only if for each reduced expression $w = r_1 \cdots r_d$, there exist $i_1 < \cdots < i_{d'}$ such that $w' = r_{i_1} \cdots r_{i_{d'}}$ is a reduced expression. In the next section, we will characterize the order in $U$ given by $u' < u$ if and only if $u' \in V(u)$.

\section{Attaching maps and the order}

We start constructing a map $\Psi_u$ from a closed cube to the Schubert cell $\S(u)$ for each $u \in U$ following similar steps used in \cite{lonardo} to build a map $\Psi_w$ from a closed cube to the Schubert cell $\S(w)$ for each $w \in W$. As in Section 1.3 of \cite{lonardo}, since each $\g_{\alpha}$ is one dimensional, there is an isomorphism from $\sl(2,\R)$ to $\g(\alpha):=\g_{\alpha} \oplus \R H_{\alpha} \oplus \g_{-\alpha}$ that can be complexified taking $\sl(2,\C)$ to $\g(\alpha)^{\C}$. Hence we can choose a normalized $E_\alpha \in \g_{\alpha}$ such that $\langle E_\alpha, \theta E_\alpha \rangle = 2\pi^2/\langle H_\alpha, H_\alpha \rangle$.

\begin{definition}
\[
\psi_j (t_j) := \exp({ F_j t_j}) \mbox{ for } t_j \in [0, 1]
\]
where $F_j = F_{\alpha_j}$ and $F_{\alpha} := E_{\alpha} + \theta E_{\alpha} \in \k$ for a fixed chosen $E_{\alpha} \in \g_{\alpha}$. Let
\[
s_{\alpha} := \exp(F_{\alpha}/2) \in U, \quad c_\alpha:= \exp(F_\alpha) \in C
\]
And similarly, $s_j = s_{\alpha_j}$, $c_j = c_{\alpha_j}$, so that ${c_j = s_j^2 = \exp (F_j)}$.
\end{definition}

Observe that $E_\alpha = \pi X_\alpha$ and $F_\alpha = \pi A_\alpha$, where $X_\alpha$ and $A_\alpha$ are defined in \cite{lonardo} with other normalization. Let $u \in U$ and $\pi (u) = w = r_1 \dots r_n$ be a fixed reduced decomposition. Note that, since $\pi(s_\alpha) =r_\alpha$, then $\pi(s_1 \cdots s_d) = r_1 \cdots r_d=w$, so that there exists $c \in C$ such that $u=s_1 \cdots s_d c$.

\begin{definition}
Let $u = s_1 \cdots s_d c $, for some $c \in C$, such that $w = \pi (u) = r_1 \cdots r_d$ is a reduced decomposition. Then define $\Psi_u : J^d \to K$,
\[
\Psi_u (t_1, \dots, t_d):= \psi_1 (t_1) \cdots \psi_d (t_d) c
\]
for $t_j \in J = [0, 1]$. Denoting $I = (0, 1)$, we have
\[
u = s_1 \cdots s_d c = \psi_1 (1/2) \cdots \psi_d (1/2) c = \Psi_u (1/2, \dots, 1/2) \in \Psi_u (I^d)
\]
\end{definition}

Now we obtain a result similiar to Proposition 1.9 of \cite{lonardo}.

\begin{proposition} \label{difeo from cube to cell}
Let $u = s_1 \cdots s_d c $, for some $c \in C$, such that $w = \pi (u) = r_1 \cdots r_d$ is a reduced decomposition. Denote the intervals $I = (0, 1) $ and $J = [0, 1]$, so that $\partial I^d = J^d \backslash I^d$, then

\[
\S (u) = \textnormal{cl} (\B (u)) = \Psi_u (J^d)
\]
The frontier of $\B (u)$ is
\[
\Psi _u \left( \partial I^d \right) = \partial \B (u) = \S (u) \backslash \B (u)
\]
and $\Psi_u |_{I^d}$ is a diffeomorphism from $I^d$ to $\B (u) = N u b$.
\end{proposition}

\begin{proof}
First we prove that $\pi:K \to \F$ is injective in each Bruhat cell in $K$. Let $n$, $n' \in N$ and assume that $n u b \neq n' u b$ and $\pi (n u b) = \pi (n' u b)$. Then there is $c \in C$ such that $n u b = n' u c b$, but since the Bruhat cells are disjoint then $c = 1$, that is a contradiction and $\pi$ is injective in each Bruhat cell.

Since $\pi$ is $G$-equivariant then $\pi (N u b) = N w b_0$, where $\pi (u) = w \in W$, so that $\pi$ is also surjective in Bruhat cells, and then bijective in Bruhat cells.

Now, from Proposition 1.9 items (1),(2),(3) of \cite{lonardo} and using that $\pi \Psi_u = \Psi_w$ then
\begin{enumerate}
\item $\pi (\Psi_u (J^d)) = \pi (\S (u) )$.

\item $\pi (\Psi_u (\textbf{t}) )\in \pi ( \partial \B (u) )$ if and only if $\textbf{t} \in \partial I^d = J^d \backslash I^d \cong S^{d-1}$.

\item $\pi \Psi_u|_{I^d} : I^d \to N w b_0 = \pi (N u b) = \pi (\B (u))$ is a diffeomorphism.
\end{enumerate}

In order to show the corresponding statements for $K$, we start with the third one. Since
\[
\pi(\Psi_u(I^d)) = \pi(Nub)
\]
we have that
\[
 \Psi_u(I^d) \subset \coprod_{c \in C} Nucb
\]
Take $N'= u' N (u')^{-1}$ where $u' = u^- u^{-1}$. In Proposition 2.7 of \cite{Patrao2}
\[
N'=(N'\cap N^-)(N'\cap N)
\]
Where in the paper the notation $N_u(B)$ is used for $N'$. Now note that
\[
u^-u^{-1}Nub= u^-u^{-1}Nu(u^-)^{-1}u^- b = N'u^-b=(N'\cap N)(N'\cap N^-)u^-b
\]
and
\[
(N'\cap N)(N'\cap N^-)u^-b =(N'\cap N)u^- (u^-)^{-1}(N'\cap N^-)u^-b = (N' \cap N)u^-b \subset Nu^-b
\]
since
\[
(u^-)^{-1}(N'\cap N^-)u^- \subset (u^-)^{-1}N^-u^- = N
\]
It follows that
\[
 Nub \subset u (u^-)^{-1} N u^-b = uN^-b
\]
since $(u^-)^{-1}Nu^- = N^-$. Hence
\[
Nucb \subset ucN^-b = uN^-cb = u (u^-)^{-1} N u^-cb = u (u^-)^{-1} \B(u^-c)
\]
Since each $\B(u^-c)$ is open and they are disjoint and since $\Psi_u(I^d)$ is connected and contains $ub$, then necessarily $\Psi_u(I^d) = Nub$.
Hence
\[
\Psi_u(J^d) = \textnormal{cl} \Psi_u (I^d) = \textnormal{cl} (\B (u)) = \S(u)
\]
Also $\psi_u(\textbf{t}) \in \partial \S (u)$ if and only if $\textbf{t} \in \partial I^d = J^d \backslash I^d \cong S^{d-1}$ and $\Psi_u|_{I^d} : I^d \to N u b = \B (u)$ is a diffeomorphism.
\end{proof}

The closed cube $J^d$ can be identified by a homeomorphism preserving orientation to a closed ball $B_d$ of dimension $d$ so that the frontier of the cube, $\partial I^d = J^d \backslash I^d$, is identified with a sphere $S^{d-1}$.

\begin{lemma} \label{lemma-order}
Let $X$ be a topological space and $\Psi: J^d \to X$ be a continuous map such that $\Psi$ is a homeomorphism from $I^d$ to $\Psi(I^d)$, that
\[
\Psi(J^d \backslash I^d) = \Psi(J^d)\backslash \Psi(I^d)
\]
For $i \in \{1,\ldots,d\}$, denote $J_i^{d-1} = J^{i-1} \times \{0, 1\} \times J^{d-i}$ and $I_i^{d-1} = I^{i-1} \times \{0, 1\} \times I^{d-i}$. If there is a subset $S$ of $\{1,\ldots,d\}$ such that, for each $i \in S$, $\Psi$ is a homeomorphism from $I_i^{d-1}$ to $\Psi(I_i^{d-1})$, that
\[
\Psi(J_i^{d-1} \backslash I_i^{d-1}) = \Psi(J_i^{d-1})\backslash \Psi(I_i^{d-1})
\]
and that
\[
\Psi(J^d \backslash I^d) = \bigcup_{i \in S} \Psi(J_i^{d-1})
\]
then
\[
\bigcup_{j \notin S} \Psi(J_j^{d-1}) \subset \bigcup_{i \in S} \Psi(J_i^{d-1})\backslash \Psi(I_i^{d-1})
\]
\end{lemma}

\begin{proof}
If $j \notin S$ and $(t_1,\ldots,t_d) \in J_j^{d-1}$, it follows that
\[
\Psi(t_1,\ldots,t_d) \in \Psi(J^d \backslash I^d) = \bigcup_{i \in S} \Psi(J_i^{d-1})
\]
It remains to show that, if $i \in S$ and $(t'_1,\ldots,t'_d) \in I_i^{d-1}$, then
\[
\Psi(t_1,\ldots,t_d) \neq \Psi(t'_1,\ldots,t'_d)
\]
Since $i \neq j$, we have that $t'_j \in I = (0,1)$ and thus there exists $\varepsilon > 0$ such that $(t'_j - \varepsilon, t'_j + \varepsilon) \subset I$. If $i < j$, denote
\[
B = I^{i-1} \times J \times I^{j-i-1} \times (t'_j - \varepsilon, t'_j + \varepsilon) \times I^{d-j}
\]
If $j < i$, denote
\[
B = I^{j-1} \times (t'_j - \varepsilon, t'_j + \varepsilon) \times I^{i-j-1} \times J \times I^{d-j}
\]
In both cases, we have that
\[
(t'_1,\ldots,t'_d) \in B \subset I^d \cup I_i^{d-1}
\]
and that $B$ is an open subset of $J^d$. Since $\Psi$ is a homeomorphism from $I^d \cup I_i^{d-1}$ to $\Psi(I^d \cup I_i^{d-1})$, it follows that $\Psi(B)$ is an open neighborhood of $\Psi(t'_1,\ldots,t'_d)$ in $\Psi(J^d)$. But we have that $t_j = 0$ or $t_j = 1$ and that
\[
\Psi(t_1,\ldots,t_d)
= \lim_{s \to 0^+} \Psi(t_1 +s_1,\ldots, t_d + s_d)
\]
where $s_k  = s$, if $t_k < 1$, and $s_k  = -s$, if $t_k = 1$. If we denote
\[
s_0 = \min \{t'_j - \varepsilon, 1 - t'_j - \varepsilon\} > 0
\]
for all $s \in (0,s_0)$, we have that
\[
(t_1 +s_1,\ldots, t_d + s_d) \in I^d \cup I_i^{d-1} \backslash B
\]
Hence $\Psi(t_1 +s_1,\ldots, t_d + s_d) \notin \Psi(B)$ and therefore $\Psi(t_1,\ldots,t_d) \notin \Psi(B)$.
\end{proof}

\begin{lemma} \label{trocando m}
For every $i$ and $j$, it holds that $c_i \psi_j (t_j) = \psi_j (t_j) c_i$, for all $t_j \in (0, 1)$, or that $c_i \psi_j (t_j) = \psi_j (1 - t_j) c_jc_i$, for all $t_j \in (0, 1)$.
\end{lemma}

\begin{proof}
Let $t_j \in (0, 1)$ and $i \neq j $, remember that $\psi_j (t_j) = \exp (t_j F_j)$, where $F_j : = E_j + \theta E_j$, where $\theta$ is the Cartan involution. First, we calculate $c_i \psi_j (t_j) c_i$. For this, let us calculate $\Ad (c_i) E_j$, since $c_i \g_{\alpha_j} = \g_{\alpha_j}$ and $c_i^2 = 1$ then $\Ad (c_i) E_j = \pm E_j$. Since $\Ad (c_i) \theta = \theta \Ad (c_i)$, then, if $\Ad (c_i) E_j = E_j$, we have that
\[
\Ad (c_i) F_j = \Ad (c_i) E_j + \theta \Ad(c_i) E_j =  E_j + \theta E_j = F_j
\]
and, if $ \Ad (c_j) E_i = - E_i$, we have that
\[
\Ad (c_i) F_j = \Ad (c_i) E_j + \theta \Ad(c_j) E_j = - E_j + \theta (- E_j) = - F_j
\]
So
\[
c_i \exp (t_j F_j) c_j^{-1} = \exp ( \Ad (c_i) (t_j F_j) ) = \exp ( \pm t_j F_j)
\]
If $\Ad (c_i) F_j = F_j$, then $c_i \psi_j (t_j) = \psi_j (t_j) c_i$, and, if $\Ad(c_i) F_j = - F_j$, then
\[
c_i \psi_j (t_j) = \psi_j (-t_j) c_i = \psi_j (1 - t_j) c_j c_i
\]
and $1 - t_i \in (0, 1)$. The case $i=j$ is trivial.
\end{proof}

\begin{lemma}\label{lemma-faces}
Let $u = s_1 \cdots s_d c $, for some $c \in C$, and denote
\[
u_i^0 = s_1 \cdots s_{i-1}s_{i+1} \cdots s_d c
\qquad
\mbox{and}
\qquad
u_i^1 = s_1 \cdots s_{i-1}c_is_{i+1} \cdots s_d c
\]
If $\pi(u) = w = r_1 \cdots r_d$ and $\pi(u_i^0) = \pi(u_i^1) = w_i = r_1 \cdots r_{i-1}r_{i+1} \cdots r_d$ are reduced expressions, then $\Psi_w$ is a diffeomorphism from $I_i^{d-1}$ to $\Psi_w(I_i^{d-1})$ with
\[
\Psi_w(I_i^{d-1}) = \B(w_i)
\qquad
\mbox{and}
\qquad
\Psi_w(J_i^{d-1} \backslash I_i^{d-1}) = \partial \B(w_i)
\]
and $\Psi_u$ is a diffeomorphism from $I_i^{d-1}$ to $\Psi_u(I_i^{d-1})$ with
\[
\Psi_u(I_i^{d-1}) = \B(u_i^0) \cap \B(u_i^1)
\qquad
\mbox{and}
\qquad
\Psi_u(J_i^{d-1} \backslash I_i^{d-1}) = \partial \B(u_i^0) \cap \partial \B(u_i^1)
\]
\end{lemma}

\begin{proof}
On one hand, it is immediate that
\[
\Psi_w (t_1, \dots, t_{i-1},0,t_{i+1}, \dots, t_d)
=
\Psi_{w_i} (t_1, \dots, t_{i-1},t_{i+1}, \dots, t_d)
\]
and that
\[
\Psi_u (t_1, \dots, t_{i-1},0,t_{i+1}, \dots, t_d)
=
\Psi_{u_i^0} (t_1, \dots, t_{i-1},t_{i+1}, \dots, t_d)
\]
On the other hand, by successive applications of Lemma \ref{trocando m}, we have that
\[
c_i\psi_{i+1}(t_{i+1}) \cdots \psi_d (t_d) = \psi_{i+1}(\phi_{i+1}(t_{i+1})) \cdots \psi_d (\phi_d(t_d)) c'
\]
where $c' \in C$ and $\phi_j(t_j) = t_j$ or $\phi_j(t_j) = 1 - t_j$. It follows that
\begin{eqnarray*}
\Psi_w (t_1, \dots, t_{i-1},1,t_{i+1}, \dots, t_d)
& = &
\psi_1 (t_1) \cdots  \psi_{i-1}(t_{i-1})c_j\psi_{i+1}(t_{i+1})   \cdots \psi_d (t_d) b_0 \\
& = &
\psi_1 (t_1) \cdots  \psi_{i-1}(t_{i-1})\psi_{i+1}(\phi_{i+1}(t_{i+1})) \cdots \psi_d (\phi_d(t_d)) c' b_0 \\
& = &
\Psi_{w_i} (t_1, \dots, t_{i-1},\phi_{i+1}(t_{i+1}), \dots, \phi_d(t_d))
\end{eqnarray*}
and that
\begin{eqnarray*}
\Psi_u (t_1, \dots, t_{i-1},1,t_{i+1}, \dots, t_d)
& = &
\psi_1 (t_1) \cdots  \psi_{i-1}(t_{i-1})c_j\psi_{i+1}(t_{i+1})   \cdots \psi_d (t_d) c \\
& = &
\psi_1 (t_1) \cdots  \psi_{i-1}(t_{i-1})\psi_{i+1}(\phi_{i+1}(t_{i+1})) \cdots \psi_d (\phi_d(t_d)) c' c \\
& = &
\Psi_{u_i^1} (t_1, \dots, t_{i-1},\phi_{i+1}(t_{i+1}), \dots, \phi_d(t_d))
\end{eqnarray*}
since
\[
u_i^1
=
\Psi_u (1/2, \dots, 1/2,1,1/2, \dots, 1/2)
=
s_1 \cdots s_{i-1}s_{i+1} \cdots s_d c' c
\]
The result follows, since
\[
(t_1, \dots, t_{i-1},t_{i+1}, \dots, t_d)
\mapsto
(t_1, \dots, t_{i-1},\phi_{i+1}(t_{i+1}), \dots, \phi_d(t_d))
\]
is a diffeomorphism from $J^{d-1}$ to itself and from $I^{d-1}$ to itself.
\end{proof}

\begin{proposition} \label{prop-order}
Let $w = r_1 \cdots r_d$ be a reduced decomposition and denote
\[
w_i = r_1 \cdots r_{i-1}r_{i+1} \cdots r_d
\]
Let $S$ be the subset of $i \in \{1,\ldots,d\}$ such that $w_i$ is a reduced expression. Then
\[
\S(w) = \B(w) \cup \bigcup_{i \in S} \S(w_i)
\]
and
\[
\bigcup_{j \notin S} \S(w_j) \subset \bigcup_{i \in S} \partial \B(w_i)
\]
Furthermore, let $u = s_1 \cdots s_d c $, for some $c \in C$, and denote
\[
u_i^0 = s_1 \cdots s_{i-1}s_{i+1} \cdots s_d c
\qquad
\mbox{and}
\qquad
u_i^1 = s_1 \cdots s_{i-1}c_is_{i+1} \cdots s_d c
\]
Then
\[
\S(u) = \B(u) \cup \bigcup_{i \in S} \S(u_i^0) \cup \S(u_i^1)
\]
and
\[
\bigcup_{j \notin S} \S(u_j^0) \cup \S(u_j^1) \subset \bigcup_{i \in S} \partial \B(u_i^0) \cup \partial \B(u_i^1)
\]
\end{proposition}

\begin{proof}
We have that
\[
\S(w)
=
\bigcup_{w'\leq w} \B(w')
=
\B(w) \cup \bigcup_{w' < w} \S(w')
\]
By the Proposition of section 5.11 of \cite{humphreys}, we have that, if $w' < w$, then $w' \leq w_i$ for some $i \in S$. Since $\S(w') \subset \S(w_i)$, if $w' \leq w_i$, it follows that
\[
\S(w)
=
\B(w) \cup \bigcup_{i \in S} \S(w_i)
\]
We have that $\Psi_w: J^d \to \F$ is a continuous map such that $\Psi_w$ is a homeomorphism from $I^d$ to $\Psi_w(I^d)$, that
\[
\Psi_w(I^{d}) = \B(w)
\qquad
\mbox{and}
\qquad
\Psi_w(J^{d}) = \S(w)
\]
and thus
\[
\Psi_w(J^d \backslash I^d) = \partial \B(w) = \Psi_w(J^d)\backslash \Psi_w(I^d)
\]
For each $i \in S$, $\Psi_w$ is a homeomorphism from $I_i^{d-1}$ to $\Psi_w(I_i^{d-1})$, that
\[
\Psi_w(I_i^{d-1}) = \B(w_i)
\qquad
\mbox{and}
\qquad
\Psi_w(J_i^{d-1}) = \S(w_i)
\]
and thus
\[
\Psi_w(J_i^{d-1} \backslash I_i^{d-1}) = \partial \B(w_i) =  \Psi_w(J_i^{d-1})\backslash \Psi_w(I_i^{d-1})
\]
Since
\[
\Psi_w(J^d \backslash I^d) =
\partial \B(w) =
\bigcup_{i \in S} \S(w_i) =
\bigcup_{i \in S} \Psi_w(J_i^{d-1})
\]
it follows that
\[
\bigcup_{j \notin S} \Psi_w(J_j^{d-1}) \subset
\bigcup_{i \in S} \Psi_w(J_i^{d-1})\backslash \Psi_w(I_i^{d-1})
=
\bigcup_{i \in S} \partial \B(w_i)
\]
Since $w_j \in \Psi_w(I_j^{d-1})$, for every $j \in S$, it follows that
\[
\bigcup_{j \notin S} \S(w_j) \subset \bigcup_{i \in S} \partial \B(w_i)
\]
Denote by $p: \partial \B(u) \to \partial \B(w)$ the restriction to $\partial \B(u)$ of $\pi: K \to \F$.
Since
\[
\partial \B(u)
\backslash
\bigcup_{i \in S} \Psi_u(I_i^{d-1})
=
\bigcup_{j \notin S} \Psi_u(J_j^{d-1})
\cup
\bigcup_{i \in S} \Psi_u(J_i^{d-1}) \backslash \Psi_u(I_i^{d-1})
\]
since
\[
p\left(\bigcup_{j \notin S} \Psi_u(J_j^{d-1}) \right)
=
\bigcup_{j \notin S} \Psi_w(J_j^{d-1})
\subset
\bigcup_{i \in S} \partial \B(w_i)
\]
and since
\[
p\left(\bigcup_{i \in S} \Psi_u(J_i^{d-1}) \backslash \Psi_u(I_i^{d-1}) \right)
=
\bigcup_{i \in S} \Psi_w(J_i^{d-1}) \backslash \Psi_w(I_i^{d-1})
=
\bigcup_{i \in S} \partial \B(w_i)
\]
it follows that
\[
p^{-1}\left(\bigcup_{i \in S} \B(w_i) \right)
=
\bigcup_{i \in S} \Psi_u(I_i^{d-1})
=
\bigcup_{i \in S} \B(u_i^0) \cup \B(u_i^1)
\]
Since $\bigcup_{i \in S} \B(w_i)$ is dense in $\partial \B(w)$ and since $\pi$ is a covering map, it follows that $\bigcup_{i \in S} \B(u_i^0) \cup \B(u_i^1)$ is dense in $\partial \B(u)$, which implies that
\[
\S(u) \backslash \B(u)
=
\partial \B(u)
=
\bigcup_{i \in S} \S(u_i^0) \cup \S(u_i^1)
\]
Finally, since
\[
p\left(\bigcup_{j \notin S} \S(u_j^0) \cup \S(u_j^1)\right)
=
\bigcup_{j \notin S} \S(w_j)
\subset
\bigcup_{i \in S} \partial \B(w_i)
\]
it follows that
\[
\bigcup_{j \notin S} \S(u_j^0) \cup \S(u_j^1) \subset \bigcup_{i \in S} \partial \B(u_i^0) \cup \partial \B(u_i^1)
\]
\end{proof}

Now the characterization of the order is almost immediate.

\begin{theorem} \label{theo-order}
Let $u, u' \in U$. If $\ell(\pi(u)) = d$ and $\ell(\pi(u')) = d-1$, then $u' < u$ if and only if for some (or equivalently for each) reduced expression $\pi(u) = r_1 \cdots r_d$ we have that $u = s_1 \cdots s_dc$, for some $c \in C$, and that
\[
u' = u_i^0 = s_1 \cdots s_{i-1}s_{i+1} \cdots s_d c
\qquad
\mbox{or}
\qquad
u' = u_i^1 = s_1 \cdots s_{i-1}c_is_{i+1} \cdots s_d c
\]
for some $1 \leq i \leq d$, where $c_i = s_i^2$. Furthermore, $u' < u$ if and only if there exists $u_1,\dots,u_n$ such that $u_1 = u'$, $u_n = u$, $\ell(\pi(u_{k+1})) = \ell(\pi(u_k)) + 1$ and $u_k < u_{k+1}$.
\end{theorem}

\section{Skeleton and boundary maps}

To construct the skeleton for the CW or cellular decomposition of $K$ obtained here, we will follow page 5 of \cite{hatcher} and construct inductively the $d$-skeleton $X^d$, or the skeleton of dimension $d$, from $X^{d-1}$ starting by the discrete set $X^0$ and attaching the open $d$-cells via {\em attaching} maps from $S^{d-1} \to X^{d-1}$. For each $u \in U$, we fix a reduced expression $w = \pi(u) = r_1 \cdots r_d$ and write $u = s_1 \cdots s_d c $ for some $c \in C$. Remember that the dimension of $\B (u)$ is thus equal to $d$. So the dimension is $0$ if and only if $l(w) = 0$ which happens if and only if $w = 1$ and thus there are $|C|$ cells of lowest dimension, given by $\B (c) = \S (c) = \{ c b \}$ for $c \in C$. In order to construct the next levels of the skeleton $X^d$, since $S^{d-1}$ is homeomorphic to the frontier of the closed cube $\partial I^d = J^d \backslash I^d$, from Proposition \ref{difeo from cube to cell} we define $\varphi_u = \Psi_u|_{\partial I^d}$. By induction on the length, we can construct all the skeleton of $K$. Note that $X^{d}$ when $d$ has maximum dimension equals $K$ and that there are $|C|$ cells of highest dimension since these correspond to the $u \in U$, such that $ \pi (u) = w^-$, the principal involution.

As a consequence of the second equation in the Proposition \ref{difeo from cube to cell}, we have the following construction. Let $d = \textrm{dim } \S (u) = \textrm{dim } \B (u) $. The sphere $S^d$ is the quotient $J^d / \partial I^d $, where the boundary $ \partial I^d = J^d \backslash I^d$ is collapsed to a point. We can do the same to the Schubert cell $\S (u)$. {\em Define}
\[
\sigma_u := \S (u) / \partial \B (u)
\]
that is, the space obtained by identifying the complement of the Bruhat cell $ \B (u)$ in $\S (u)$ to a point. Since $\Psi_u$ is a diffeomorphism from $I^d$ to $\B (u)$ and since $\Psi_u ( \partial I^d ) \subset \partial \B (u) $, it follows that $\Psi_u$ induces a map $S^d \to \sigma_u$ which is a homeomorphism. The inverse of this homeomorphism will be denoted by
\begin{equation} \label{inverse}
\Psi_u^{-1} : \sigma_u \to S^d
\end{equation}
note that this is not the same as the inverse of $\Psi_u$ because of the collapse in the corresponding borders.

The following is Proposition 1.10 of \cite{lonardo} and the remark after it for Schubert cells in $\F$ and it will be useful later on. Remember that all roots have multiplicity one for split real forms.

\begin{proposition}
Let $w$, $w' \in W$. The following statements are equivalent:

\begin{enumerate}
\item $\S(w') \subset \S(w) \subset \F $ and ${ \rm dim} \, \S(w) - { \rm dim} \, \S(w') = 1$.

\item If $w = r_1 \cdots r_d$ is a reduced expression of $w \in W$ as a product of simple reflections and for some $i$, $w' = r_1 \cdots r_{i-1}r_{i+1} \cdots r_d$ is also a reduced expression.
\end{enumerate}
In this case, there is a unique integer $i$ such that $w' = r_1 \cdots r_{i-1}r_{i+1} \cdots r_d$.
\end{proposition}

Let $\pi (u) = w = r_1 \cdots r_d$ and let $v \in U$ such that $\pi (v) < \pi (u)$ in the Bruhat order and the length of $\pi (v)$ is equal to $d-1$ then by Theorem 5.10 of \cite{humphreys}, $w' = \pi (v) = r_1 \cdots \widehat{r_i} \cdots r_d$ for a unique integer $i$ as the next Lemma shows. First, note that this expression for $\pi (v)$ is necessarily reduced, since the length of $\pi (v)$ is $d-1$.

Let $\mathcal{C}$ be the $\Z$-module freely generated by $\B (u)$, $u \in U$. The {\em boundary maps} $\delta : \mathcal{C} \to \mathcal{C}$ are defined by
\[
\delta \B (u) := \sum_{v \in V (u)} \rho (u, v) \B (v)
\]
where $\rho (u, v) := 0$, if ${\rm dim } \B (u) - {\rm dim } \B (v) \neq 1$ and
\[
\rho (u, v) := {\rm deg} \left( \phi_{u, v}: S_u^{d-1} \to S_v^{d-1} \right)
\]
if ${ {\rm dim} \B (u) - {\rm dim} \B (v) = 1}$, where the map $\phi_{u, v}$ is the {\em composition of the following maps}:

\vspace{4 mm}

(a) the attaching map $\Psi_u|_{\partial I^d} : S_u^{d-1} \cong \partial I^d   \to \partial \B (u) = \S (u) \backslash \B (u)$.

\vspace{3 mm}

(b) the quotient map $\partial \B (u) \to \partial \B (u)  / (\partial \B (u) \backslash \B (v) ) $, where we take the cell $\B (v)$ and identify its complement in $\partial \B (u) = \S (u) \backslash \B (u) $ to a point.

\vspace{3 mm}

(c) the identification $ \partial \B (u)  / (\partial \B (u) \backslash \B (v) ) \cong \S (v) / \partial \B (v) = \sigma_v$, where the last equality comes from the definition of $\sigma_v$.

\vspace{3 mm}

(d) the map defined by equation (\ref{inverse}), $\Psi_v^{-1} : \sigma_v \to S_v^{d-1}$.

\vspace{4 mm}

{\em Remark }

\vspace{3 mm}

We need to compute the degree
\[
\rho (u, v) = {\rm deg} \left( \phi_{u, v}: S_u^{d-1} \to S_v^{d-1} \right)
\]
when $u = s_1 \cdots s_d c $, for some $c \in C$, and
\[
v = u_i^0 = s_1 \cdots s_{i-1}s_{i+1} \cdots s_d c
\qquad
\mbox{or}
\qquad
v = u_i^1 = s_1 \cdots s_{i-1}c_is_{i+1} \cdots s_d c
\]
where $\pi(u) = w = r_1 \cdots r_d$ and $\pi(u_i^0) = \pi(u_i^1) = w_i = r_1 \cdots r_{i-1}r_{i+1} \cdots r_d$ are reduced expressions. We will determine the degree of the map $\phi_{u, v}$ in three steps.

\vspace{3 mm}

{\em Step 1: Domain and co-domain spheres.}

\vspace{3 mm}

First we identify the spheres $S_u^{d-1}$ in the domain and $S_v^{d-1}$ in the co-domain.
Remember that the ``closed ball'', $B^d_u = J^d$, where $J = [0, \pi]$, as in Theorem \ref{difeo from cube to cell} and the domain of $\phi _{u, v}$ is the frontier of $J^d$:
\[
S_u^{d-1} = \{ (t_1, \dots , t_d): t_i = 0 \mbox{ or } \pi \mbox{ for some } i \}
\]
or the union of the ``faces'' of the closed cube $J^d$. On the other hand, let $B^{d-1}_v = J^{d-1}$. The co-domain is the sphere $S_v^{d-1}$ obtained by collapsing to a point the boundary of $B^{d-1}_v = J^{d-1}$. This is seen in the items (c) and (d) in the definition of $\phi_{u, v}$.

\vspace{3 mm}

{\em Step 2: $\sigma_v$ in the image $\Psi_u (S_u^{d-1})$.}

\vspace{3 mm}

The second step is to see how $\sigma_v$ sits inside the image $\Psi_u (S_u^{d-1})$. Lemma \ref{lemma-faces} shows what is the pre-image of $N v b$ under $\Psi_u$. The pre-image of $\B (v) $ is contained in $B^d_u$ and is the interior of one of the two faces corresponding to the $i$th coordinate, that is, the faces where $t_i = 0$ in the case $v = u_i^0$ or $t_i = 1$ in the case $v = u_i^1$. In the quotient $\sigma_v = \S (v) / \partial \B (v) $ all the faces of $\partial B^d_u$ corresponding to the $j$th coordinates $j \neq i$ collapse to a point.

\vspace{3 mm}

{\em Step 3: Degrees.}

\vspace{3 mm}

The degree of a map can be computed as the local degree in the inverse image of of a regular value which has a finite number of points (see \cite{hatcher}, Proposition 2.30). The degree of $\phi_{u, v}$ is then the local degree of
\[
f_i^0 = {\Psi}_{u_i^0}^{-1} \circ \Psi_u|_{I_i^0}  : I_i^0 \to I^{d-1}
\]
where
\[
I_i^0 = I^{i-1} \times \{0\} \times I^{d-i}
\]
in the case $v = u_i^0$ or the local degree of
\[
f_i^1 = {\Psi}_{u_i^1}^{-1} \circ \Psi_u|_{I_i^1}  : I_i^1 \to I^{d-1}
\]
where
\[
I_i^1 = I^{i-1} \times \{1\} \times I^{d-i}
\]
in the case $v = u_i^1$. The maps $f_i^0$ and $ f_i^1$ are diffeomorphisms so their local degrees are $\pm 1$.

\section{Algebraic expressions for the degrees}

Here we follow closely \cite{lonardo}in order to compute the coefficients $\rho(u, v)$ in terms of the roots by finding the local degrees of the maps involved. For a diffeomorphism $\varphi$ of the sphere its degree is the local degree at a point $x$ which in turn is the sign of the determinant $\det d \varphi_x$ with respect to a volume form of $S^d$. Let us apply this in our context. We let $u = s_1 \cdots s_d c$ and $v = u_i^0$ or $v = u_i^1$. We must find the degree of $f_i^0$ or the degree of $f_i^1$. In the definition of $f_i^0$ and $f_i^1$ appears, respectively, the maps $\Psi_{u_i^0}$ and $\Psi_{u_i^1}$, which are given through a reduced expression of $w_i = \pi(u_i^0) = \pi(u_i^1)$ which was fixed in the construction of the skeleton and which, in principle, can be different from the reduced expression $w_i = r_1 \cdots \widehat{r_i} \cdots r_d$. On the other hand, the reduced expression $w_i = r_1 \cdots \widehat{r_i} \cdots r_d$ can be used to define another characteristic maps, which will be denoted by $\Phi_{u_i^0}$ and $\Phi_{u_i^1}$. This new characteristic map can then be used to define new functions
\[
p_i^0 = {\Phi}_{u_i^0}^{-1} \circ \Psi_u|_{I_i^0}  : I_i^0 \to I^{d-1}
\]
in the case $v = u_i^0$ and
\[
p_i^1 = {\Phi}_{u_i^1}^{-1} \circ \Psi_u|_{I_i^1}  : I_i^1 \to I^{d-1}
\]
in the case $v = u_i^1$. The two pairs of functions are then related by
\[
f_i^{\epsilon} = \left( {\Psi}_{u_i^\epsilon}^{-1} \circ {\Phi}_{u_i^\epsilon} \right) \circ p_i^{\epsilon}, \quad \epsilon = 0,1
\]
The composition ${\Psi}_{u_i^\epsilon}^{-1} \circ {\Phi}_{u_i^\epsilon}$ (also understood as a map between spheres in which the boundary is collapsed to points) is an homeomorphism of spheres and, hence, has degree $\pm 1$. The term ${\Psi}_{u_i^\epsilon}^{-1} \circ {\Phi}_{u_i^\epsilon}$ can disappear if in the graph of the Bruhat order it is possible to make so that every word immediately bellow is a subword of the word above. When the Bruhat order graph grows more complicated this is not possible, see \cite{lonardo2}.

Before getting these degrees we make the following discussion on the orientation of the faces of the cube $[-1,1]^d$, centered on the origin of $\R^d$, which is given with the basis $\{ e_1,\dots, e_d \}$. Afterwards we will use the cube $[0,1]^d$, since there is a translation and cagnification taking one into the other, we have corresponding orientations. Starting with the $(d-1)$-dimensional sphere $S^{d-1}$ we orient the tangent space at $x \in S^{d-1}$ by a basis $\{f_2, \dots , f_d \}$ such that $\{x, f_2, \dots , f_d \}$ is positively oriented. The faces of $B = [-1,1]^d$ are oriented accordingly: given a base vector $e_i$, we let $B_i^-$ be the face perpendicular to $e_i$ that contains $-e_i$ and $B_i^+$ be the one that contains $e_i$. Then $B_i^-$ has the same orientation as the basis $e_1, \dots, \widehat{e_i}, \dots, \e_d$ if $i$ is even (since $-e_i, e_1, \dots, \widehat{e_i}, \dots, e_d$ is positively oriented if $i$ is even). So that the orientation of $B_i^-$ is $(-1)^i$ the orientation of $e_1, \dots, \widehat{e_i}, \dots, e_d$. And similarly the orientation of $B_i^+$ is $(-1)^{i+1}$ the orientation of $e_1, \dots, \widehat{e_i}, \dots, e_d$. In the case of one dimension the point $1$ is positive oriented and $-1$ is negative oriented, in two dimensions the sides of the squares are oriented counter-clockwise and in the cube the faces are oriented following the right hand rule, we will call these orientations the { \em standard} orientation.

The following facts, which is Lemma 2.4 from \cite{lonardo} with minor changes, will be used below in the computation of the degrees.

\begin{lemma} \label{diferencial do m}
For a root $\alpha$ consider the action on $K$ of $c = c_\alpha = \exp( F_{\alpha})$. Then
\begin{enumerate}
\item $c v b = v v^{-1} c v b = v c' b$, for $c' = v^{-1} c v \in M$, and $c N c^{-1} = N$

\item the restriction of $c$ to $N v b = \B(v)$ is a diffeomorphism from $\B(v)$ to $\B(v c')$.

\item the differential $d c_{v b}$ identifies to the action of $c$ restricted to the subspace
\[
\sum_{\beta \in \Pi_{\pi(v)}} \g_\beta
\]

\end{enumerate}
\end{lemma}

\begin{proof}
From Lemma \ref{acao do m} for $\beta \in \Pi$ then $c_{\alpha} \g_\beta = \g_\beta$. Since
\[
c N v b = c N c ^{-1} c v b = N c v b = N v v^{-1} c v b = N v c' b
\]
where $c' = v^{-1} c v \in M$.

For the third statement we use the notation $X \cdot k = d/dt (\exp (t X))|_{t=0}$, for $k \in K$ and $X \in \g$. Also, for $A \subset \g$ let $A \cdot k = \{ X \cdot k : X \in A$ \}.

Note that $N v b = v v^{-1} N v b$ and the tangent space to $v^{-1} N v b$ at $b$ is spanned by $\g_\gamma \cdot b$ with $\gamma < 0$ such that $\gamma = v^{-1} \beta$ and $ \beta >0$, that is, $v \gamma >0$. Since $(d v) (\g_\gamma \cdot b) = \g_{v \gamma} \cdot v b$, it follows that the tangent space $T_{v b} (N v b)$ is spanned by $\g_\beta \cdot b$ where $\beta = v \gamma >0$ such that $v^{-1} \beta = \gamma <0$. Hence the result.
\end{proof}

And the following is Lemma 1.8 from \cite{lonardo}.

\begin{lemma} \label{acao do m}
The root spaces $\g_{\beta}$ are invariant by the action of $c_\alpha$ and
\[
(c_\alpha)|_{\g_\beta} = (-1)^{\epsilon(\alpha,\beta)} {\rm id}
\]
where
\[
\epsilon(\alpha,\beta)
=
\frac{2 \langle \alpha , \beta \rangle} {\langle \alpha, \alpha \rangle}
\]
is an integer.
\end{lemma}

The following is Proposition 2.5 from \cite{lonardo} only separating the degrees, changing notation, and observing that $I = i$, since in our case the multiplicities of roots are always $1$.

\begin{proposition} \label{degrees}
The degrees are:
\begin{enumerate}
\item ${\rm deg}(p_i^0) = (-1)^i$.

\item ${\rm deg}(p_i^1) = (-1)^{i+1+\sigma}$, where
\end{enumerate}
\begin{equation}\label{eq:sigma}
\sigma = \sum_{\beta \in \Pi_{w_i'}} \frac{2\langle \alpha_i, \beta \rangle}{\langle \alpha_i, \alpha_i \rangle}
\end{equation}
with $\Pi_{w_i'} = \Pi^+ \cap w_i' \Pi^-$ and $w_i' = r_{i+1} \cdots r_d$.
\end{proposition}

\begin{proof}
On one hand, the map $p_i^0$ is the projection of the face of a $d$-dimensional cube onto the face of a $d-1$-dimensional cube given by
\[
p_i^0(t_1, \dots, t_{i-1}, 0 , t_{i+1}, \dots, t_d) = (t_1, \dots, t_{i-1}, t_{i+1}, \dots t_d)
\]
Note that with respect to the basis $e_1, \dots, e_d$ the $t_i$-coordinate appears in the $i$th position. Hence, by the orientation of the cube, if $i$ is even or odd, the orientation is respectively, positive or negative. Therefore, ${\rm deg} (p_i^0) = (-1)^i$.

On the other hand, the map $p_i^1$ is given by
\[
p_i^1(t_1, \dots, t_{i-1}, 1 , t_{i+1}, \dots, t_d) = {\Phi}_{u_i^1}^{-1}(\psi_1(t_1) \cdots \psi_{i-1}(t_{i-1})c_i \psi_{i+1}(t_{i+1}) \cdots \psi_d(t_d))
\]
where $c_i = \psi_i(1)$. The left action of $c_i$ on $K$ takes any Bruhat cell $\B(v)$ to the cell $\B(v c_i')$, where $c_i' = v^{-1} c_i v$.
And hence it takes the Schubert cell $\S(v)$ to $\S(vc_i')$. Moreover the restriction of $c_i$ to $N v b$ is a diffeomorphism.

In particular, we restrict the action of $c_i$ to the cell $\S(u_i')$, $u_i' = s_{i+1} \cdots s_d c$. Its action on $K$ takes the Bruhat cell $\B(u_i')$ to the cell $\B(u_i' c_i')$ where $c_i' = u_i'^{-1} c_i u_i'$. By successive applications of Lemma \ref{trocando m}, we have that
\[
c_i \psi_{i+1}(t_{i+1}) \cdots \psi_d (t_d) cb = \psi_{i+1}(\phi_{i+1}(t_{i+1})) \cdots \psi_d (\phi_d(t_d)) c_i'c b
\]
where $\phi_j(t_j) = t_j$ or $\phi_j(t_j) = 1 - t_j$. We can define $\overline{c}_i : J^{d-i} \to J^{d-i}$ by
\[
\overline{c}_j (t_{j+1},\dots, t_d) =
(\phi_{i+1}(t_{i+1}), \dots, \phi_d(t_d))
\]
which is continuous and a diffeomorphism in $I^{d-i}$. Hence, $p_i^1 (t_1, \dots,1, \dots , t_d)$ becomes the projection of the $(i-1)$ first coordinates and the composition of $\overline{c}_i$ with the projection of the last $k$th coordinates, $k = i+1,\dots, d$. From the choice of the orientation of $J^d$, the face $(t_1, \dots, 1, \dots, t_d)$ of $J^d$ has orientation $(-1)^{i+1}$ with respect to the orientation of the coordinates $(t_1, \dots, t_{i-1}, t_{i+1}, \dots, t_d)$. Hence after collapsing the boundary to a point, we get the degree
\[
{\rm deg}\, p_i^1 = (-1)^{i+1} {\rm deg}\, \overline{c}_i
\]
The degree of $\overline{c}_i$ equals its degree at one point which in turn is the sign of the determinant of the differential $d(c_i)_{u_i'b}$ restricted to the tangent space of the Bruhat cell $\B(u_i') = N u_i' b$ at $u_i'b$:
\[
{\rm deg}\, p_i^1 = (-1)^{i+1} {\rm sgn} \left[ {\rm det} \, \left( d(c_i)_{u_i' b} |_{T_{u_i' b} (N u_i' b)} \right) \right]
\]
By the third statement in Lemma \ref{diferencial do m} $T_{u_i' b} (N u_i' b)$ identifies to $\sum_{\beta \in \Pi_{u_i'}} \g_\beta $. Once we have the generators $\g_\beta \cdot u_i' b$, $\beta \in \Pi_{u_i'}$ for $T_{u_i' b} (N u_i' b)$ together with the action of $c_i$ over $\g_\beta$ given by Lemma \ref{acao do m}, $c_\alpha |_{\g_{\beta}} = (-1)^{\epsilon (\alpha, \beta)} {\rm id}$, we conclude that the signal of ${\rm det} \, \left( d(c_i)_{u_i' b} |_{T_{u_i' b} (N u_i' b)} \right) = (-1)^\sigma$ where
\[
\sigma=\sum_{\beta \in \Pi_{u_i'}} \frac{2 \langle \alpha_i , \beta \rangle} {\langle \alpha_i, \alpha_i \rangle}
\]
\end{proof}

Summarizing, we have the following algebraic expression for the coefficient $\rho(u, v)$.

\begin{theorem} \label{degrees1}
Let $\sigma$ be as in Equation \ref{eq:sigma}. Then
\[
\rho(u, v) = {\rm deg} \left( {\Psi}_{u_i^0}^{-1} \circ {\Phi}_{u_i^0} \right) (-1)^i
\]
when $v = u_i^0$, and
\[
\rho(u, v) = {\rm deg} \left( {\Psi}_{u_i^1}^{-1} \circ {\Phi}_{u_i^1} \right) (-1)^{i + 1 + \sigma}
\]
when $v = u_i^1$.
\end{theorem}

\section{Examples}

\subsection{$G = \Sl (3)$ and $K = \SO (3)$}

Since right multiplication by $c \in C$ is a diffeomorphism then $\B (u c) = \B (u) c $ for $u \in U$. So $\delta (\B (u c)) = \delta (\B (u)) c $. So to obtain $\delta (\B (u))$ for all $u \in U$ we need only obtain $\delta (\B (u))$ for the $5$ elements $s_1$, $s_2$, $s_1 s_2 $, $s_2 s_1 $, $s_1 s_2 s_1$. Let $E_{i,j}$ be the matrix with $\pi$ in the position $(i,j)$ and zero elsewhere. Take $A = E_{1,2} - E_{2,1}$ and $B = E_{2,3}-E_{3,2}$ note that $\psi_1(t)=e^{tA}$ and $\psi_2(s)=e^{sB}$ then
\begin{flalign*}
&\Psi_1(0)= b & \\
&\Psi_{s_1}(t) = e^{tA} b, \, t \in [0,1] &\\
&\Psi_{s_2}(s) = e^{sB} b, \, s \in [0,1] &\\
&\Psi_{s_1 s_2}(t,s) = e^{tA} e^{sB} b, \, (t,s) \in [0,1]^2 &\\
&\Psi_{s_2 s_1} (t,s)= e^{tB} e^{sA} b, \, (t,s) \in [0,1]^2 &\\
&\Psi_{s_1 s_2 s_1} (t,s,z)= e^{tA} e^{sB} e^{zA} b, \, (t,s,z) \in [0,1]^3 &\\
\end{flalign*}

The multiples by $c$ to the right are similar. Then we obtain expressions for $\rho$. The following calculations can be done more geometrically by comparing orientations of the maps with the standard orientation or more algebraically by calculating $\sigma(u,v)$ as in Proposition \ref{degrees1}.  The Bruhat diagram for the Weyl group can be represented as:

\begin{center}
\begin{tikzpicture}[node distance={21mm}, thick, main/.style = {draw, rectangle} ]
\node[main] (1) {$r_1 r_2 r_1$};
\node[main] (2) [below left of=1] {$r_1 r_2$};
\node[main] (3) [below right of=1] {$r_2 r_1$};
\node[main] (4) [below of=2] {$r_1$};
\node[main] (5) [below of=3] {$r_2$};
\node[main] (6) [below right of=4] {$1$};
\draw[->] (1) -- (2);
\draw[->] (1) -- (3);
\draw[->] (2) -- (4);
\draw[->] (2) -- (5);
\draw[->] (3) -- (5);
\draw[->] (3) -- (4);
\draw[->] (4) -- (6);
\draw[->] (5) -- (6);
\end{tikzpicture}
\end{center}

Note that in this case there is no need to account for the factor ${\Psi}_{u_i^\epsilon}^{-1} \circ {\Phi}_{u_i^\epsilon}$ since there is only one expression needed for each of the five elements. More complex cases need more care, see \cite{lonardo2}.

First calculating geometrically:

\begin{enumerate}
\item $\rho(s_1,c_1)=1$ and $\rho(s_1,1)=-1$ since $f_1^1 (1)=e^{A} = c_1$ and $f_1^0 (0)=1$. Remember the convention we adopted for orientation of the cube $[-1,1]^d$: at the final point in the direction of the axis it is $+1$ and in the initial point opposite the axis it is $-1$.

\item $\rho(s_2,c_2)=1$ and $\rho(s_2,1)=-1$ since $f_1^1 (1)=e^{B} = c_2$ and $f_1^0 (0)=1$.

\item $\rho(s_1 s_2,s_1 c_2) = -1$, $\rho(s_1 s_2,s_1) = 1$, $\rho(s_1 s_2,c_1 s_2) = -1$,$\rho(s_1 s_2,s_2) = -1$. We need to consider the degree of 4 maps: $f_2^1(t,1) = e^{tA} e^{B} = e^{tA} c_2$, $f_2^0 (t,0)= e^{tA}1 = e^{tA}$, $f_1^1 (1,s)= e^{A} e^{sB} = c_1 e^{sB} = e^{(1-s)B} c_2 c_1$ since $c_1 e^{sB} c_1 = e^{-sB}$, $f_1^0 (0,s)= 1 e^{sB} = e^{sB}$. These maps are illustrated with orientations in the next picture. Note that following our convention the positive orientation in this case is the counter-clockwise orientation.

\begin{center}
\begin{tikzpicture}
	[square/.style={very thick,black}, axis/.style={->,blue,thick}]

\usetikzlibrary {arrows.meta}

% \fill (2.5,-1) circle (0pt) node[below=0pt] {$\B(s_1 s_2)$};

\fill (2,2) circle (2.2pt) node[below=2pt] {$s_1 s_2$};

\fill (2,0) circle (2pt) node[below=2pt] {$s_1$};

\fill (4,2) circle (2pt) node[right=2pt] {$c_1 s_2 = s_2 c_1 c_2$};

\fill (2,4) circle (2pt) node[above=2pt] {$s_1 c_2$};

\fill (0,2) circle (2pt) node[left=2pt] {$s_2 $};

\draw [-{Classical TikZ Rightarrow[length=2.5mm]}] (0,0) node[anchor=north]{$1 \quad$} -- (1.25,0);
\draw [-{Classical TikZ Rightarrow[length=2.5mm]}] (0,0) -- (0,1) ;
\draw [-{Classical TikZ Rightarrow[length=2.5mm]}] (4,4) -- (4,3) ;
\draw [-{Classical TikZ Rightarrow[length=2.5mm]}] (0,4) -- (1,4) ;

\draw [-{Classical TikZ Rightarrow[length=2.5mm]}] (2,0) -- (3,0);
\draw [-{Classical TikZ Rightarrow[length=2.5mm]}] (0,2) -- (0,3) ;
\draw [-{Classical TikZ Rightarrow[length=2.5mm]}] (4,2) -- (4,1) ;
\draw [-{Classical TikZ Rightarrow[length=2.5mm]}] (2,4) -- (3,4) ;

\draw[axis] (0,0) -- (5,0) node[anchor=north]{$ \ \ x$};
\draw[axis] (0,0) -- (0,5) node[anchor=east]{$y \,$};

\draw[square] (0,0) -- (0,4) node[anchor=east]{$ c_2 $}-- (4,4) node[anchor=west]{$ c_1 c_2$} -- (4,0) node[anchor=north]{$ c_1 $} -- cycle;
\end{tikzpicture}
\end{center}

\item $\rho(s_2 s_1,s_2 c_1) = -1$, $\rho(s_2 s_1,s_2) = 1$, $\rho(s_2 s_1,c_2 s_1) = -1$,$\rho(s_2 s_1,s_1) = -1$. We need to consider the degree of 4 maps: $f_2^1(t,1) = e^{tB} e^{A} = e^{tB} c_1$, $f_2^0 (t,0)= e^{tB}1 = e^{tB}$, $f_1^1 (1,s)= e^{B} e^{sA} = c_2 e^{sA} = e^{(1-s)A} c_2 c_1$ since $c_2 e^{sA} c_2 = e^{-sA}$, $f_1^0 (0,s)= 1 e^{sB} = e^{sB}$. These maps are illustrated with orientations in the next picture.

\begin{center}
\begin{tikzpicture}
	[square/.style={very thick,black}, axis/.style={->,blue,thick}]

\usetikzlibrary {arrows.meta}

% \fill (2.5,-1) circle (0pt) node[below=0pt] {$\B(s_2 s_1)$};

\fill (2,2) circle (2.2pt) node[below=2pt] {$s_2 s_1$};

\fill (2,0) circle (2pt) node[below=2pt] {$s_2$};

\fill (4,2) circle (2pt) node[right=2pt] {$c_2 s_1 = s_1 c_1 c_2$};

\fill (2,4) circle (2pt) node[above=2pt] {$s_2 c_1$};

\fill (0,2) circle (2pt) node[left=2pt] {$s_1 $};

\draw [-{Classical TikZ Rightarrow[length=2.5mm]}] (0,0) node[anchor=north]{$1 \quad$} -- (1,0);
\draw [-{Classical TikZ Rightarrow[length=2.5mm]}] (0,0) -- (0,1) ;
\draw [-{Classical TikZ Rightarrow[length=2.5mm]}] (4,4) -- (4,3) ;
\draw [-{Classical TikZ Rightarrow[length=2.5mm]}] (0,4) -- (1,4) ;

\draw [-{Classical TikZ Rightarrow[length=2.5mm]}] (2,0) -- (3,0);
\draw [-{Classical TikZ Rightarrow[length=2.5mm]}] (0,2) -- (0,3) ;
\draw [-{Classical TikZ Rightarrow[length=2.5mm]}] (4,2) -- (4,1) ;
\draw [-{Classical TikZ Rightarrow[length=2.5mm]}] (2,4) -- (3,4) ;

\draw[axis] (0,0) -- (5,0) node[anchor=north]{$ \ \ x$};
\draw[axis] (0,0) -- (0,5) node[anchor=east]{$y \,$};

\draw[square] (0,0) -- (0,4) node[anchor=east]{$ c_1 $}-- (4,4) node[anchor=west]{$ c_1 c_2$} -- (4,0) node[anchor=north]{$ c_2 $} -- cycle;

\end{tikzpicture}
\end{center}

\item $\rho(s_1 s_2 s_1,s_1 s_2 c_1)=1$, $\rho(s_1 s_2 s_1,s_1 s_2)=-1$, $\rho(s_1 s_2 s_1,c_1 s_2 s_1)=1$, ${\rho(s_1 s_2 s_1,s_2 s_1)=-1}$. We need to consider 4 maps:
$f_2^1(t,s,1) = e^{tA} e^{sB} c_1$, $f_2^0 (t,s,0)= e^{tA}e^{sB}$, $f_1^1 (1,s,z)= c_1 e^{sB} e^{zA} = e^{(1-s)B} c_1 c_2 e^{zA} = e^{(1-s)B} e^{(1-z)A} c_2$, since $c_1 e^{sB} c_1 = e^{-sB}$ and $c_2 e^{zA} c_2 = e^{-zA}$ , $f_1^0 (0,s,z)= e^{sB} e^{zA}$. These 4 maps are illustrated in the next two pictures with orientations of the faces on the edges.
\end{enumerate}

\begin{center}
\begin{tikzpicture}
	[cube/.style={very thick,black}, axis/.style={->,blue,thick}]

\usetikzlibrary {arrows.meta}

\fill [fill=cyan!]  (0,0,0) -- (0,4,0) -- (4,4,0) -- (4,0,0) ;

\fill [fill=cyan!]  (0,0,4) -- (0,4,4) -- (4,4,4) -- (4,0,4) ;

\fill (2,2,0) circle (1.8pt) node[below=2pt] {$s_1 s_2$};

\fill (2,2,2) circle (2.2pt) node[below=2pt] {$s_1 s_2 s_1$};

\fill (2,2,4) circle (1.8pt) node[right=2pt] {$s_1 s_2 c_1$};

 \draw [-{Classical TikZ Rightarrow[length=2.5mm]}] (0,0,0) -- (2,0,0);
 \draw [-{Classical TikZ Rightarrow[length=2.5mm]}] (0,4,0) -- (0,2,0);
 \draw [-{Classical TikZ Rightarrow[length=2.5mm]}] (4,0,0) -- (4,2,0);
 \draw [-{Classical TikZ Rightarrow[length=2.5mm]}] (4,4,0) -- (2,4,0);

 \draw [-{Classical TikZ Rightarrow[length=2.5mm]}] (0,0,4) -- (2,0,4);
 \draw [-{Classical TikZ Rightarrow[length=2.5mm]}] (0,4,4) -- (0,2,4);
 \draw [-{Classical TikZ Rightarrow[length=2.5mm]}] (4,0,4) -- (4,2,4);
 \draw [-{Classical TikZ Rightarrow[length=2.5mm]}] (4,4,4) -- (2,4,4);

%draw the axes
\draw[axis] (0,0,0) -- (5,0,0) node[anchor=west]{$x$};
\draw[axis] (0,0,0) -- (0,5,0) node[anchor=west]{$\, y$};
\draw[axis] (0,0,0) -- (0,0,6) node[anchor=east]{$z$};

%draw the top and bottom of the cube

\draw[cube] (0,0,0) -- (0,4,0) node[anchor=east]{$ c_2 \ $} -- (4,4,0) node[anchor=west]{$ c_1 c_2$} -- (4,0,0) -- cycle;
\draw[cube] (0,0,4) node[anchor=east]{$ c_1 $ }-- (0,4,4) node[anchor=east]{$ c_1 c_2 $} -- (4,4,4) node[anchor=west]{$ \ c_2 $ }-- (4,0,4) node[anchor=north]{$1$} -- cycle;

%draw the edges of the cube
\draw[cube] (0,0,0) -- (0,0,4);
\draw[cube] (0,4,0) -- (0,4,4);
\draw[cube] (4,0,0) -- (4,0,4);
\draw[cube] (4,4,0) -- (4,4,4);
\end{tikzpicture}
\end{center}

\begin{center}
\begin{tikzpicture}
	[cube/.style={very thick,black}, axis/.style={->,blue,thick}]

\usetikzlibrary {arrows.meta}

\fill [fill=orange!95!black]  (0,0,0) -- (0,4,0) -- (0,4,4) -- (0,0,4) ;

\fill [fill=orange!95!black]  (4,0,0) -- (4,4,0) -- (4,4,4) -- (4,0,4) ;

\fill (0,2,2) circle (1.8pt) node[below=2pt] {$s_2 s_1$};

\fill (2,2,2) circle (2.2pt) node[below=2pt] {$s_1 s_2 s_1$};

\fill (4,2,2) circle (1.8pt) node[below=2pt] {$c_1 s_2 s_1$};

 \draw [-{Classical TikZ Rightarrow[length=2.5mm]}] (0,0,0) -- (0,2,0);
 \draw [-{Classical TikZ Rightarrow[length=2.5mm]}] (0,4,0) -- (0,4,2);
 \draw [-{Classical TikZ Rightarrow[length=2.5mm]}] (0,4,4) -- (0,2,4);
 \draw [-{Classical TikZ Rightarrow[length=2.5mm]}] (0,0,4) -- (0,0,2);
 \draw [-{Classical TikZ Rightarrow[length=2.5mm]}] (4,0,0) -- (4,2,0);
 \draw [-{Classical TikZ Rightarrow[length=2.5mm]}] (4,4,0) -- (4,4,2);
 \draw [-{Classical TikZ Rightarrow[length=2.5mm]}] (4,4,4) -- (4,2,4);
 \draw [-{Classical TikZ Rightarrow[length=2.5mm]}] (4,0,4) -- (4,0,2);

%draw the axes
\draw[axis] (0,0,0) -- (5,0,0) node[anchor=west]{$x$};
\draw[axis] (0,0,0) -- (0,5,0) node[anchor=west]{$\, y$};
\draw[axis] (0,0,0) -- (0,0,6) node[anchor=east]{$z$};

%draw the top and bottom of the cube

\draw[cube] (0,0,0) -- (0,4,0) node[anchor=east]{$ c_2 \ $} -- (4,4,0) node[anchor=west]{$ c_1 c_2$} -- (4,0,0) -- cycle;
\draw[cube] (0,0,4) -- (0,4,4) node[anchor=east]{$ c_1 c_2 $} -- (4,4,4) -- (4,0,4) node[anchor=north]{$1$} -- cycle;

%draw the edges of the cube
\draw[cube] (0,0,0) -- (0,0,4);
\draw[cube] (0,4,0) -- (0,4,4);
\draw[cube] (4,0,0) -- (4,0,4);
\draw[cube] (4,4,0) -- (4,4,4);
\end{tikzpicture}
\end{center}

Now we obtain the same results using the more general algebraic method:

Remember we need to find $\Pi_{w_i'} = \Pi^+ \cap w_i' \Pi^-$ for each case to find the respective $\sigma(u,v)$. For this we use equation \ref{vara} and that $\alpha_1 = (1,-1,0)$ and $\alpha_2 = (0,1,-1)$. Here we will use $*$ to represent $c_j$ or $1$.
\begin{enumerate}
\item To calculate $\rho(s_1,*)$ we have $i=1$ and $w_1'=1$ so $\Pi_{1} = \varnothing$. Then $\rho(s_1,1)=(-1)^1$ and $\rho(s_1,c_1) = (-1)^{1+1}$

\item For $\rho(s_2,*)$ we have $i=1$ and $w_1'=1$ so $\Pi_{1} = \varnothing$. Then $\rho(s_2,1)=(-1)^1$ and $\rho(s_2,c_1) = (-1)^{1+1}$

\item For $\rho(s_1s_2,s_1 *)$ we have $i=2$ and $w_2'=1$ so $\Pi_{1} = \varnothing$. Then $\rho(s_1 s_2,s_1)=(-1)^2$ and $\rho(s_1 s_2,s_1 c_2) = (-1)^{2+1}$. Now for $\rho(s_1s_2,* s_2)$ we have $i=1$ and $w_1'=r_2$ so $\Pi_{r_2} =\{ \alpha_2\}$ by Equation \ref{vara} and $\sigma=(2.(-1))/2 = -1$. Then
$\rho(s_1 s_2, s_2) = (-1)^1$ and $\rho(s_1 s_2,c_1 s_2) = (-1)^1$.

\item For $\rho(s_2s_1,s_2 *)$ we have $i=2$ and $w_2'=1$ so $\Pi_{1} = \varnothing$. Then $\rho(s_2 s_1,s_2)=(-1)^2$ and $\rho(s_2 s_1,s_2 c_1) = (-1)^{2+1}$. Now for $\rho(s_2s_1,* s_1)$ we have $i=1$ and $w_1'=r_1$ so $\Pi_{r_1} =\{ \alpha_1\}$ by Equation \ref{vara} and $\sigma=(2.(-1))/2 = -1$. Then
$\rho(s_2 s_1, s_2) = (-1)^1$ and $\rho(s_2 s_1,c_1 s_2) = (-1)^1$.

\item For $\rho(s_1s_2s_1,s_1s_2 *)$ we have $i=3$ and $w_3'=1$ so $\Pi_{1} = \varnothing$. Then $\rho(s_1s_2 s_1,s_1s_2)=(-1)^3$ and $\rho(s_1s_2 s_1,s_1s_2 c_1) = (-1)^{3+1}$. Now for $\rho(s_1 s_2 s_1,* s_2 s_1)$ we have $i=1$ and $w_1'=r_2 r_1$ so $\Pi_{r_2 r_1} =\{ \alpha_2,r_2 \alpha_1 \} = \{ \alpha_2 ,\alpha_1 +\alpha_2 \}$ by Equation \ref{vara} and
\[
\sigma = \frac{2\langle \alpha_1 ,\alpha_2 \rangle }{\langle \alpha_1, \alpha_1 \rangle} + \frac{2 \langle \alpha_1, \alpha_1 +\alpha_ 2\rangle }{\langle \alpha_1, \alpha_1 \rangle} = \frac{2.(-1)}{2} + \frac{2.(1)}{2} = 0
\]
so $\rho(s_1 s_2 s_1,s_2 s_1) = (-1)^1$ and $\rho(s_1 s_2 s_1,c_1 s_2 s_1) = (-1)^{1+1} = 1$
\end{enumerate}
Writing $\B (c)$ as $c$, for any $c \in C$ we get,
\begin{flalign*}
&\delta_1 \B (s_1) = c_1 - 1 &\\
&\delta_1 \B (s_2) = c_2 - 1 &\\
\\
&\delta_2 \B (s_1 s_2) = \B (s_1) (1 - c_2) - \B (s_2) (1 + c_1 c_2) &\\
&\delta_2 \B (s_2 s_1) = \B (s_2) (1 - c_1) - \B (s_1) (1 + c_1 c_2) &\\
\\
&\delta_3 \B (s_1 s_2 s_1) = \B (s_1 s_2) (c_1 - 1) + \B (s_2 s_1) (c_2 - 1) &
\end{flalign*}

Calculating the kernel and image of boundary maps $\delta_k$ we can then calculate the homology of the compact group $K$.
\[
H_k = \frac {\ker \delta_k}{\Ima \delta_{k+1}}
\]
The calculation of homology groups of $K=\SO(3)$ using the kernels and images of $\delta_k$ is done in the Appendix. We obtained:
\[
H_0(K)=\Z; \quad H_1(K)=\Z/2\Z = \Z_2; \quad H_2(K) =0 ; \quad H_3(K) = \Z
\]
These results agree with $\SO(3)$ being homeomorphic to the projective three dimensional space.

\subsection{$G=G_2$ and $K=\SO(4)$}

Consider the group $G_2$ as the split real form of Dynkin diagram $G_2$, sometimes denoted $G_{2,2}$. It can be seen by the Iwasawa decomposition that the dimension of the maximal compact subgroup is $6$, Yokota showed in 1977 that the maximal compact subgroup is in fact $\SO(4)$ (see \cite{yokota}). Let the two simple roots of the split real algebra $\g_2$ be $\{ \alpha, \beta \}$ then
\[
\Pi^+=\{\alpha,\beta,\alpha+\beta,2 \alpha +\beta,3\alpha + \beta, 3 \alpha +2 \beta \}
\]
are the positive roots, let $r_1$ correspond to $\alpha$ and $r_2$ correspond to $\beta$, then the corresponding Weyl group is
\[
W  = \{1, r_1,r_2,r_1 r_2,r_2 r_1,r_1r_2r_1,r_2r_1r_2,(r_1r_2)^2,(r_2 r_1)^2,r_1(r_2r_1)^2,r_2 (r_1 r_2)^2,(r_1r_2)^3 \}
\]
Note that $(r_1r_2)^3  =(r_2r_1)^3$ since $(r_1r_2)^6=1$ and the Bruhat order for the Weyl group of $G_2$ is

\begin{center}
\begin{tikzpicture}[node distance={21mm}, thick, main/.style = {draw, rectangle} ]
\node[main] (1) {$(r_1 r_2)^3$};
\node[main] (2) [below left of=1] {$(r_1 r_2)^2 r_1$};
\node[main] (3) [below right of=1] {$(r_2 r_1)^2 r_2$};
\node[main] (4) [below of=2] {$(r_1 r_2)^2$};
\node[main] (5) [below of=3] {$(r_2 r_1)^2$};
\node[main] (6) [below of=4] {$r_1 r_2 r_1$};
\node[main] (7) [below of=5] {$r_2 r_1 r_2$};
\node[main] (8) [below of=6] {$r_1 r_2$};
\node[main] (9) [below of=7] {$r_2 r_1 $};
\node[main] (10) [below of=8] {$r_1 $};
\node[main] (11) [below of=9] {$r_2 $};
\node[main] (12) [below right of=10] {$1$};

\draw[->] (1) -- (2);
\draw[->] (1) -- (3);
\draw[->] (2) -- (4);
\draw[->] (2) -- (5);
\draw[->] (3) -- (4);
\draw[->] (3) -- (5);
\draw[->] (4) -- (6);
\draw[->] (4) -- (7);
\draw[->] (5) -- (6);
\draw[->] (5) -- (7);
\draw[->] (6) -- (8);
\draw[->] (6) -- (9);
\draw[->] (7) -- (8);
\draw[->] (7) -- (9);
\draw[->] (8) -- (10);
\draw[->] (8) -- (11);
\draw[->] (9) -- (10);
\draw[->] (9) -- (11);
\draw[->] (10) -- (12);
\draw[->] (11) -- (12);
\end{tikzpicture}
\end{center}
Note that,
\[
s_1 c_2 s_1^{-1} =\exp\left(\frac{1}{2} F_1\right) \exp (i  H_2^{\vee}) \exp \left(\frac{1}{2} F_1\right)^{-1} = \exp \left(\Ad (\exp \left(\frac{1}{2} F_1\right)) i H_2^{\vee}\right)
\]
equals
\[
\exp \left(e^{\ad (\frac{1}{2} F_1)} i H_2^{\vee}\right) = \exp (r_1 (i H_2^{\vee}))
\]
This comes from the equation before Lemma 1.7 in \cite{lonardo} and
\[
\exp (r_1 (i H_2^{\vee})) = \exp (i r_1 (H_2^{\vee}))= \exp (i r_{\alpha} ({H_\beta}^{\vee}))
\]
now
\[
r_{\alpha} ({H_\beta}^{\vee}) = H_\beta^{\vee} - 2\frac{\langle H_\alpha^{\vee},H_\beta^{\vee} \rangle}{\langle H_\alpha^{\vee},H_\alpha^{\vee} \rangle } H_\alpha^{\vee} = H_\beta^{\vee} - 2\frac{\langle \alpha, \beta\rangle}{\langle \alpha,\alpha \rangle } H_\alpha^{\vee}= H_\beta^{\vee} - 2\left(\frac{-3}{2}\right) H_\alpha^{\vee}
\]
so
\[
r_{\alpha} ({H_\beta}^{\vee})= H_\beta^{\vee} +3 H_\alpha^{\vee} = H_2^{\vee} +3 H_1^{\vee}
\]
and
\[
\exp (r_1 (\pi i H_2^{\vee})) = \exp (\pi i (H_2^{\vee} +3 H_1^{\vee}))=c_2c_1^3=c_1c_2
\]
so
\[
s_1 c_2 s_1^{-1}= c_1 c_2
\]
and, through similar calculations we also obtain,
\[
s_2 c_1 s_2^{-1}= c_1 c_2
\]

Note that $s_1 c_2 s_1^{-1}= c_1 c_2$ then by conjugating by $s_1^{-1}$ on both sides we get $c_2 = s_1^{-1} c_1 c_2 s_1$ so that $c_2 = c_1 s_1^{-1} c_2 s_1$ and
\[
s_1 c_1 c_2  = c_2 s_1
\]
from $s_2 c_1 s_2^{-1}= c_1 c_2$ we similarly get
\[
s_2 c_1 c_2  = c_1 s_2
\]

Remember we need to find $\Pi_{v_2} = \Pi^+ \cap v_2 \Pi^-$ for each case to find the respective $\sigma(u,v)$. For this we use equation \ref{vara} and that $\Pi_{v_2} = \Pi_{\pi(v_2)}$ take $\alpha_1=\alpha = (\frac{1}{2},-\frac{\sqrt{3}}{2})$ and $\alpha_2=\beta = (0,\sqrt{3})$. Here we will use $*$ to represent $c_j$ or $1$. We will also use Proposition \ref{degrees} to calculate $\rho(u,v)$.

\begin{enumerate}
\item To calculate $\rho(s_1,*)$ we have $i=1$ and $w_1'=1$ so $\Pi_{1} = \varnothing$. Then $\rho(s_1,1)=(-1)^1$ and $\rho(s_1,c_1) = (-1)^{1+1}$.

\item For $\rho(s_2,*)$ we have $i=1$ and $w_1'=1$ so $\Pi_{1} = \varnothing$. Then $\rho(s_2,1)=(-1)^1$ and $\rho(s_2,c_2) = (-1)^{1+1}$.

\item For $\rho(s_1s_2,s_1 *)$ we have $i=2$ and $w_2'=1$ so $\Pi_{1} = \varnothing$. Then $\rho(s_1 s_2,s_1)=(-1)^2$ and $\rho(s_1 s_2,s_1 c_2) = (-1)^{2+1}$. Now for $\rho(s_1s_2,* s_2)$ we have $i=1$ and $w_1'=r_2$ so $\Pi_{r_2} =\{ \alpha_2\}$ by Equation \ref{vara} and $\sigma=2(-3/2)/1 = -3$. Then
$\rho(s_1 s_2, s_2) = (-1)^1$ and $\rho(s_1 s_2,c_1 s_2) = (-1)^{1+1-3}$. Note that $c_1s_2 = s_2 c_1 c_2$.

\item For $\rho(s_2s_1,s_2 *)$ we have $i=2$ and $w_2'=1$ so $\Pi_{1} = \varnothing$. Then $\rho(s_2 s_1,s_2)=(-1)^2$ and $\rho(s_2 s_1,s_2 c_1) = (-1)^{2+1}$. Now for $\rho(s_2s_1,* s_1)$ we have $i=1$ and $w_1'=r_1$ so $\Pi_{r_1} =\{ \alpha_1\}$ by Equation \ref{vara} and $\sigma=2(-3/2)/3 = -1$. Then
$\rho(s_2 s_1, s_1) = (-1)^1$ and $\rho(s_2 s_1,c_2 s_1) = (-1)^1$. Note that $c_2s_1 = s_1 c_1 c_2$.

\item For $\rho(s_1s_2s_1,s_1s_2 *)$ we have $i=3$ and $w_3'=1$ so $\Pi_{1} = \varnothing$. Then $\rho(s_1s_2 s_1,s_1s_2)=(-1)^3$ and $\rho(s_1s_2 s_1,s_1s_2 c_1) = (-1)^{3+1}$. Now for $\rho(s_1 s_2 s_1,* s_2 s_1)$ we have $i=1$ and $w_1'=r_2 r_1$ so $\Pi_{r_2r_1} =\{ \alpha_2,r_2 \alpha_1 \} = \{ \alpha_2 ,\alpha_1 +\alpha_2 \}$ by Equation \ref{vara} and
\[
\sigma = \frac{2\langle \alpha_1 ,\alpha_2 \rangle }{\langle \alpha_1, \alpha_1 \rangle} + \frac{2 \langle \alpha_1, \alpha_1 +\alpha_ 2\rangle }{\langle \alpha_1, \alpha_1 \rangle} = \frac{2.(-3/2)}{1} + \frac{2.(-1/2)}{1} = -4
\]
so $\rho(s_1 s_2 s_1,s_2 s_1) = (-1)^1$ and $\rho(s_1 s_2 s_1,c_1 s_2 s_1) = (-1)^{1+1-4} = 1$. Note that $c_1 s_2 s_1 = s_2 c_1 c_2 s_1 = s_2 c_1 s_1 c_1 c_2= s_2 s_1 c_2$.

\item For $\rho(s_2s_1s_2,s_2s_1 *)$ we have $i=3$ and $w_3'=1$ so $\Pi_{1} = \varnothing$. Then $\rho(s_2s_1 s_2,s_2s_1)=(-1)^3$ and $\rho(s_2s_1 s_2,s_2s_1 c_2) = (-1)^{3+1}$. Now for $\rho(s_2 s_1 s_2,* s_1 s_2)$ we have $i=1$ and $w_1'=r_1 r_2$ so $\Pi_{r_1r_2} =\{ \alpha_1,r_1 \alpha_2 \} = \{ \alpha_1 ,3\alpha_1 +\alpha_2 \}$ by Equation \ref{vara} and
\[
\sigma = \frac{2\langle \alpha_2 ,\alpha_1 \rangle }{\langle \alpha_2, \alpha_2 \rangle} + \frac{2 \langle \alpha_2, 3\alpha_1 +\alpha_ 2\rangle }{\langle \alpha_2, \alpha_2 \rangle} = \frac{2.(-3/2)}{3} + \frac{2.(-9/2 + 3)}{3} = -2
\]
so $\rho(s_2 s_1 s_2,s_1 s_2) = (-1)^1$ and $\rho(s_2 s_1 s_2,c_2 s_1 s_2) = (-1)^{1+1-2} = 1$. Note that $c_2 s_1 s_2 = s_1 c_1 c_2 s_2 = s_1 c_1 s_2 c_2 = s_1 s_2 c_1 c_2 c_2 = s_1 s_2 c_1$.

\item For $\rho((s_1s_2)^2,s_1s_2s_1 *)$ we have $i=4$ and $w_4'=1$ so $\Pi_{1} = \varnothing$. Then $\rho((s_1s_2)^2,s_1s_2s_1)=(-1)^4$ and $\rho((s_1s_2)^2,s_1s_2 s_1 c_2) = (-1)^{4+1}$. Now for $\rho((s_1s_2)^2,* s_2 s_1s_2)$ we have $i=1$ and $w_1'=r_2 r_1 r_2$ so
\[
\Pi_{r_2 r_1 r_2} =\{ \alpha_2,r_2 \alpha_1 , r_2 r_1 \alpha_2 \} = \{ \alpha_2 ,\alpha_1 +\alpha_2 , r_2(\alpha_2+3\alpha_1)\} =
\]
\[
= \{ \alpha_2 ,\alpha_1 +\alpha_2 , 3\alpha_1+2\alpha_2\}
\]
by Equation \ref{vara} and
\[
\sigma = \frac{2\langle \alpha_1 ,\alpha_2 \rangle }{\langle \alpha_1, \alpha_1 \rangle} + \frac{2 \langle \alpha_1, \alpha_1 +\alpha_ 2\rangle }{\langle \alpha_1, \alpha_1 \rangle} + \frac{2 \langle \alpha_1, 3\alpha_1 +2\alpha_ 2\rangle }{\langle \alpha_1, \alpha_1 \rangle}
\]
\[
\sigma=2.(-3/2) + 2.(-1/2)+ 2.(3+2.(-3/2))= -4
\]
so $\rho((s_1s_2)^2,s_2 s_1s_2) = (-1)^1$ and $\rho((s_1s_2)^2,c_1 s_2 s_1s_2) = (-1)^{1+1-4} = 1$. Note that $c_1 s_2 s_1 s_2= s_2 c_1 c_2 s_1 s_2 = s_2 c_1 s_1 c_1 c_2 s_2 = s_2 s_1 c_2 s_2= s_2 s_1 s_2 c_2$.

\item For $\rho((s_2s_1)^2,s_2s_1s_2 *)$ we have $i=4$ and $w_4'=1$ so $\Pi_{1} = \varnothing$. Then $\rho((s_2s_1)^2,s_2s_1s_2)=(-1)^4$ and $\rho((s_2s_1)^2,s_2s_1 s_2 c_1) = (-1)^{4+1}$. Now for $\rho((s_2 s_1)^2 ,* s_1 s_2s_1)$ we have $i=1$ and $w_1'=r_1 r_2 r_1$ so
\[
\Pi_{r_1 r_2 r_1} =\{ \alpha_1,r_1 \alpha_2 , r_1 r_2 \alpha_1 \} = \{ \alpha_1 ,3\alpha_1 +\alpha_2 , r_1(\alpha_1 +\alpha_2)\} =
\]
\[
= \{ \alpha_1 ,3\alpha_1 +\alpha_2 , 2\alpha_1+\alpha_2\}
\]
by Equation \ref{vara} and
\[
\sigma = \frac{2\langle \alpha_2 ,\alpha_1 \rangle }{\langle \alpha_2, \alpha_2 \rangle} + \frac{2 \langle \alpha_2, 3\alpha_1 +\alpha_ 2\rangle }{\langle \alpha_2, \alpha_2 \rangle} + \frac{2 \langle \alpha_2, 2\alpha_1 +\alpha_ 2\rangle }{\langle \alpha_2, \alpha_2 \rangle} =
\]
\[
\quad=\frac{2.(-3/2)}{3} + \frac{2.(-9/2+3)}{3}+ \frac{2.(-6/2 +3))}{3}= -4
\]
so $\rho((s_2 s_1)^2,s_1 s_2 s_1) = (-1)^1$ and $\rho((s_2 s_1)^2,c_2 s_1 s_2 s_1) = (-1)^{1+1-4} = 1$. Note that $c_2 s_1 s_2 s_1 = s_1 c_1 c_2 s_2 s_1 = s_1 c_1 s_2 c_2 s_1 = s_1 s_2 c_1 c_2 c_2 s_1 = s_1 s_2 c_1 s_1 = s_1 s_2 s_1 c_1$.

\item For $\rho((s_1s_2)^2 s_1,(s_1s_2)^2 )$ we have $i=5$ and $w_5'=1$ so $\Pi_{1} = \varnothing$. Then $\rho((s_1s_2)^2 s_1,(s_1s_2)^2=(-1)^5$ and $\rho((s_1s_2)^2s_1,(s_1s_2)^2 c_1) = (-1)^{5+1}$. Now for $\rho((s_1s_2)^2 s_1,* (s_2 s_1)^2)$ we have $i=1$ and $w_1'=(r_2 r_1)^2$ so
\[
\Pi_{(r_2 r_1)^2} =\{ \alpha_2,r_2 \alpha_1 , r_2 r_1 \alpha_2 ,r_2r_1r_2 \alpha_1\} =
\]
\[
\quad = \{ \alpha_2 ,\alpha_1 +\alpha_2 , 3\alpha_1+2\alpha_2, 2\alpha_1+\alpha_2 \}
\]
by Equation \ref{vara} and
\[
\sigma = \frac{2\langle \alpha_1 ,\alpha_2 \rangle }{\langle \alpha_1, \alpha_1 \rangle} + \frac{2 \langle \alpha_1, \alpha_1 +\alpha_ 2\rangle }{\langle \alpha_1, \alpha_1 \rangle} + \frac{2 \langle \alpha_1, 3\alpha_1 +2\alpha_ 2\rangle }{\langle \alpha_1, \alpha_1 \rangle} + \frac{2 \langle \alpha_1, 2\alpha_1 +\alpha_ 2\rangle }{\langle \alpha_1, \alpha_1 \rangle} =
\]
\[
\quad =2.(-3/2) + 2.(-1/2)+ 2.(3+2.(-3/2))+2.(2+(-3/2)) = -3
\]
so $\rho((s_1s_2)^2 s_1,(s_2 s_1)^2 = (-1)^1$ and $\rho((s_1s_2)^2s_1,c_1 (s_2 s_1)^2 = (-1)^{1+1-3} = -1$. Note by calculations similar to item 5, $c_1 (s_2 s_1)(s_2 s_1)= (s_2 s_1) c_2 (s_2 s_1) = (s_2 s_1) (s_2 s_1) c_1 c_2$.

\item For $\rho((s_2s_1)^2 s_2,(s_2s_1)^2 )$ we have $i=5$ and $w_5'=1$ so $\Pi_{1} = \varnothing$. Then $\rho((s_2s_1)^2 s_2,(s_2s_1)^2=(-1)^5$ and $\rho((s_2s_1)^2 s_2,(s_2s_1)^2 c_2) = (-1)^{5+1}$. Now for $\rho((s_2s_1)^2 s_2,* (s_1 s_2)^2)$ we have $i=1$ and $w_1'=(r_1 r_2)^2$ so
\[
\Pi_{(r_1 r_2)^2} =\{ \alpha_1,r_1 \alpha_2 , r_1 r_2 \alpha_1 ,r_1r_2r_1 \alpha_2\} =
\]
\[
\quad = \{ \alpha_1 ,3\alpha_1 +\alpha_2 , 2\alpha_1+\alpha_2, 3\alpha_1+2\alpha_2 \}
\]
by Equation \ref{vara} and
\[
\sigma = \frac{2\langle \alpha_2 ,\alpha_1 \rangle }{\langle \alpha_2, \alpha_2 \rangle} + \frac{2 \langle \alpha_2, 3\alpha_1 +\alpha_ 2\rangle }{\langle \alpha_2, \alpha_2 \rangle} + \frac{2 \langle \alpha_2, 2\alpha_1 +\alpha_ 2\rangle }{\langle \alpha_2, \alpha_2 \rangle} + \frac{2 \langle \alpha_2, 3\alpha_1 +2\alpha_ 2\rangle }{\langle \alpha_2, \alpha_2 \rangle} =
\]
\[
\quad =\frac{2.(-3/2)}{3} + \frac{2.(-9/2 +3)}{3}+ \frac{2.(-3 +3)}{3}+ \frac{2.(-9/2+6))}{3} = -1
\]
so $\rho((s_1s_2)^2 s_1,(s_2 s_1)^2 = (-1)^1$ and $\rho((s_1s_2)^2s_1,c_1 (s_2 s_1)^2 = (-1)^{1+1-1} = -1$. Note by calculations similar to item 6, $c_2 (s_1 s_2)(s_1 s_2) = (s_1 s_2) c_1 (s_1 s_2) = (s_1 s_2) (s_1 s_2) c_1 c_2$.

\item For $\rho((s_1s_2)^3,(s_1s_2)^2 s_1 *)$ we have $i=6$ and $w_6'=1$ so $\Pi_{1} = \varnothing$. Then $\rho((s_1s_2)^3,(s_1s_2)^2s_1)=(-1)^6$ and $\rho((s_1s_2)^3,(s_1s_2)^2 s_1 c_2) = (-1)^{6+1}$. Now for $\rho((s_1s_2)^3,* (s_2 s_1)^2 s_2)$ we have $i=1$ and $w_1'=(r_2 r_1)^2 r_2$ so
\[
\Pi_{(r_2 r_1)^2 r_2} =\{ \alpha_2,r_2 \alpha_1 , r_2 r_1 \alpha_2 , r_2 r_1 r_2 \alpha_1 , r_2 r_1 r_2 r_1 \alpha_2 \} =
\]
\[
\quad = \{ \alpha_2 ,\alpha_1 +\alpha_2 , 3\alpha_1+2\alpha_2,2\alpha_1+\alpha_2,3\alpha_1+\alpha_2 \}
\]
by Equation \ref{vara} and putting the term $\frac{2\langle \alpha_1, \cdot \rangle }{\langle \alpha_1, \alpha_1 \rangle}$ in evidence,
\[
\sigma = \frac{2\langle \alpha_1 ,( \alpha_2 + (\alpha_1 +\alpha_2) + (3\alpha_1+2\alpha_2) + (2\alpha_1+\alpha_2) + (3\alpha_1+\alpha_2)) \rangle }{\langle \alpha_1, \alpha_1 \rangle} =
\]
\[
\quad=2\langle \alpha_1,(9\alpha_1 + 6\alpha_2) \rangle = 2\left( 9 - 6.\left( \frac{-3}{2}\right) \right)=0
\]
so $\rho((s_1s_2)^3,(s_2 s_1)^2s_2) = (-1)^1$ and $\rho((s_1s_2)^2,c_1 (s_2 s_1)^2 s_2) = (-1)^{1+1+0} = 1$. Note by calculations similar to item 9, $ c_1 (s_2 s_1)^2 s_2 = (s_2 s_1)^2 c_1 c_2 s_2 = (s_2 s_1)^2 c_1 s_2 c_2 = (s_2 s_1)^2  s_2 c_1$.
\end{enumerate}

Writing $\B (c)$ as $c$, for any $c \in C$ we get,
\begin{flalign*}
&\delta_1 \B (s_1) = c_1 - 1 &\\
&\delta_1 \B (s_2) = c_2 - 1 &\\
\\
&\delta_2 \B (s_1 s_2) = \B (s_1) (1 - c_2) - \B (s_2) (1 + c_1 c_2) &\\
&\delta_2 \B (s_2 s_1) = \B (s_2) (1 - c_1) - \B (s_1) (1 + c_1 c_2) &\\
\\
&\delta_3 \B (s_1 s_2 s_1) = \B (s_1 s_2) (c_1 - 1) + \B (s_2 s_1) (c_2 - 1) & \\
&\delta_3 \B (s_2 s_1 s_2) = \B (s_2 s_1) (c_2 - 1) + \B (s_1 s_2) (c_1 - 1) &
\\
\\
&\delta_4 \B ((s_1 s_2)^2) = \B (s_1 s_2 s_1) (1- c_2 ) + \B (s_2 s_1s_2) (c_2 - 1) &\\
&\delta_4 \B ((s_2 s_1)^2) = \B (s_2 s_1 s_2) (1 -c_1 ) + \B (s_1 s_2 s_1) (c_1 - 1) &\\
\\
&\delta_5 \B ((s_1 s_2)^2 s_1) = \B ((s_1 s_2)^2) (c_1 - 1) - \B ((s_2 s_1)^2) (1+ c_1c_2) &\\
&\delta_5 \B ((s_2 s_1)^2 s_2) = \B ((s_2 s_1)^2) (c_2 - 1) - \B ((s_1 s_2)^2) (1+ c_1c_2) &\\
\\
&\delta_6 \B ((s_1 s_2)^3) = \B ((s_1 s_2)^2 s_1) (1 - c_2) + \B ((s_2 s_1)^2 s_2) (c_1 - 1) &
\end{flalign*}

The calculation of homology groups of $K=\SO(4)$ using the kernels and images of $\delta_k$ turn impractical to be done by hand, we then express the $\delta_k$ in matrix form so we can use the Sagemath homology package to calculate the homology groups more easily. We obtained:
\[
H_0(K)=\Z; \quad H_1(K)=\Z_2; \quad H_2(K) =0 ; \quad H_3(K) = \Z \oplus \Z;
\]
\[
H_4(K)= \Z_2; \quad H_5(K)=0; \quad H_6(K)= \Z
\]

These are indeed the homology groups of $\SO(4)$, one can alternatively find the homology groups of $\SO(4)$ by using that $\SO(4)$ is homeomorphic to $S^3 \times \SO(3)$ and applying, for example, Theorem 3B.6 in \cite{hatcher}. The calculation is relatively easy since the ``Tor" components are trivial.

\appendix

\section{Appendix}

Since the operators $\delta_k$ are equivariant with relation to right multiplication with members of $M$ then the kernels and images in the previous expression also are invariant by right multiplication by $M$. Note also that $\delta_k$ are all linear operators so in the matrix format is always possible to find the kernels and images by linear algebra calculations. In the following calculations we will write $\B(s_j s_k)$ and $\B(s_j)$ as $s_j s_k$ and $s_j$, respectively. Also, let $a_i \in \Z$.

\textbf{Calculating} $\ker \delta_1$:
\[
\delta_1(s_1 (a_1+a_2 c_1+a_3 c_2+a_4 c_1 c_2)+s_2(a_5+a_6 c_1+a_7 c_2+a_8 c_1 c_2))
\]
\[
=(c_1-1) (a_1+a_2 c_1+a_3 c_2+a_4 c_1 c_2)+(c_2-1)(a_5+a_6 c_1+a_7 c_2+a_8 c_1 c_2)
\]
\[
=(-a_1+a_2-a_5+a_7)+c_1(a_1-a_2-a_6+a_8)+c_2(-a_3+a_4+a_5-a_7)+c_1c_2(-a_4+a_3+a_6-a_8)
\]
So to get $\ker \delta_1$ we need
\[
\left\{\begin{array}{llll}
a_1+a_5=a_2+a_7 \\
a_1+a_8=a_2+a_6 \\
a_3+a_7=a_4+a_5 \\
a_3+a_6=a_4+a_8 \\
\end{array}\right.
\]
Subtracting the second line by the first and the fourth line by the third,
\[
\left\{\begin{array}{llll}
a_1+a_5=a_2+a_7 \\
a_5+a_6=a_7+a_8 \\
a_3+a_7=a_4+a_5 \\
a_5+a_6=a_7+a_8 \\
\end{array}\right.
\]
So
\[
\left\{\begin{array}{llll}
a_1=a_2-a_5+a_7 = a_2 +(a_6-a_7-a_8)+a_7 = a_2+a_6-a_8\\
a_5=-a_6+a_7+a_8 \\
a_3=a_4+a_5-a_7 = a_4 +(-a_6+a_7+a_8)-a_7 = a_4-a_6+a_8\\
\end{array}\right.
\]
where the $a_5$ was substituted in the first and third line, so that
\[
s_1 (a_1+a_2 c_1+a_3 c_2+a_4 c_1 c_2)+s_2(a_5+a_6 c_1+a_7 c_2+a_8 c_1 c_2)
\]
\[
= s_1 ((a_2+a_6-a_8)+a_2 c_1+(a_4-a_6+a_8) c_2+a_4 c_1 c_2)
\]
\[
+s_2((-a_6+a_7+a_8)+a_6 c_1+a_7 c_2+a_8 c_1 c_2)
\]
reorganizing the terms
\[
a_2 s_1(1+c_1) +a_4 s_1(c_2+c_1c_2) +a_6(s_1(1-c_2)+s_2(-1+c_1))
\]
\[
+a_7(s_2(1+c_2)) +a_8 (s_1(-1+c_2)+s_2(1+c_1c_2))
\]
In the previous expression the terms multiplying $a_2$, $a_4$, $a_6$, $a_7$, $a_8$ are then generators of $\ker \delta_1$. To simplify adding the term from $a_6$ to $a_8$ we get that:
\[
\ker \delta_1 = \langle s_1(1+c_1),s_1 c_2(1+c_1), s_2(1+c_2), s_2 c_1(1+c_2),s_1(1-c_2)-s_2(1-c_1) \rangle
\]

\textbf{Calculating} $\ker \delta_2$:
\[
\delta_2(s_1 s_2 (a_1+a_2 c_1+a_3 c_2+a_4 c_1 c_2)+s_2s_1(a_5+a_6 c_1+a_7 c_2+a_8 c_1 c_2))
\]
\[
=\left.\begin{array}{lll}
(s_1(1-c_2) - s_2(1+c_1c_2))(a_1+a_2 c_1+a_3 c_2+a_4 c_1 c_2) \\
+ (s_2(1-c_1)-s_1(1+c_1c_2))(a_5+a_6 c_1+a_7 c_2+a_8 c_1 c_2)
\end{array}\right.
\]
\[
=\left.\begin{array}{lll}
s_1((a_1-a_3-a_5-a_8)+ c_1(a_2-a_4-a_6-a_7)+c_2(a_3-a_1-a_6-a_7) \\
+c_1c_2(a_4-a_2-a_5-a_8))+s_2((a_5-a_1-a_4-a_6)+ c_1(a_6-a_2-a_3-a_5) \\
+c_2(a_7-a_2-a_3-a_8)+c_1c_2(a_8-a_1-a_4-a_7))
\end{array}\right.
\]
So to get $\ker \delta_2$ we need
\[
\left\{\begin{array}{llllllll}
a_1=a_3+a_5+a_8 \\
a_2=a_4+a_6+a_7 \\
a_3=a_1+a_6+a_7 = (a_3+a_5+a_8)+a_6+a_7 \\
a_4=a_2+a_5+a_8 = (a_4+a_6+a_7)+a_5+a_8 \\
a_5=a_1+a_4+a_6 = (a_3+a_5+a_8)+a_4+a_6 \\
a_6=a_2+a_3+a_5 = (a_4+a_6+a_7)+a_3+a_5 \\
a_7=a_2+a_3+a_8 = (a_4+a_6+a_7)+a_3+a_8 \\
a_8=a_1+a_4+a_7 = (a_3+a_5+a_8)+a_4+a_7 \\
\end{array}\right.
\]
where we substituted the the first two lines in the other equations. With some cancellations we note that the fourth, the seventh and eighth are redundant, then
\[
\left\{\begin{array}{lllll}
a_1=a_3+a_5+a_8 \\
a_2=a_4+a_6+a_7 \\
a_5+a_8+a_6+a_7 =0\\
a_3+a_8+a_4+a_6 =0\\
a_4+a_7+a_3+a_5 =0\\
\end{array}\right.
\]
Summing the last three equations we get $a_3+a_4+a_6+a_8+a_5+a_7 = 0$ and substituting the third, fourth and fifth equation we get:
\[
\left\{\begin{array}{lllll}
a_1=a_3+a_5+a_8 \\
a_2=a_4+a_6+a_7 \\
a_3+a_4=0\\
a_5+a_7=0\\
a_6+a_8=0\\
\end{array}\right.
\]
With this we can put every term as function of $a_4$,$a_7$,$a_8$ as
\[
\left\{\begin{array}{lllll}
a_1=-a_4-a_7+a_8 \\
a_2=a_4-a_8+a_7 \\
a_3=-a_4\\
a_5=-a_7\\
a_6=-a_8\\
\end{array}\right.
\]
So that
\[
s_1 s_2 (a_1+a_2 c_1+a_3 c_2+a_4 c_1 c_2)+s_2 s_1(a_5+a_6 c_1+a_7 c_2+a_8 c_1 c_2)
\]
\[
=\left.\begin{array}{lll}
s_1 s_2 ((-a_4-a_7+a_8)+(a_4-a_8+a_7)c_1-a_4 c_2 + a_4 c_1 c_2) \\
+s_2 s_1(-a_7-a_8 c_1+a_7 c_2+a_8 c_1 c_2)
\end{array}\right.
\]
reorganizing terms,
\[
=\left.\begin{array}{lll}
a_4 s_1 s_2 (-1+c_1-c_2+c_1c_2)+ a_7(s_1 s_2(-1+c_1)+ s_2 s_1(-1+c_2)) \\
+a_8(s_1 s_2(1-c_1)+s_2 s_1(-c_1+c_1c_2))
\end{array}\right.
\]
In the previous expression the three terms multiplying $a_4$,$a_7$,$a_8$ are then a generating set for $\ker \delta_2$ or
\[
\ker \delta_2 =\left.\begin{array}{lll}
\langle s_1 s_2 (c_1-1)(c_2 + 1), s_1 s_2 (c_1-1) + s_2 s_1(c_2-1), \\
\, \, s_1 s_2 (c_1-1) + s_2 s_1 c_1(1-c_2) \rangle &
\end{array}\right.
\]

\textbf{Calculating} $\ker \delta_3$:
\begin{flalign*}
& \delta_3 (s_1 s_2 s_1(a_1+a_2 c_1+a_3 c_2+a_4 c_1 c_2)) = &\\
& \quad \quad =(s_1 s_2 (c_1 - 1) + s_2 s_1(c_2 - 1))(a_1+a_2 c_1+a_3 c_2+a_4 c_1 c_2)
\end{flalign*}
\[
=\left.\begin{array}{lll}
s_1 s_2 ((a_2 - a_1) +(a_1 - a_2)c_1 +(a_4-a_3)c_2 +(a_3-a_4)c_1c_2) \\
 +s_2 s_1((a_3 - a_1) +(a_4-a_2)c_1 +(a_1 - a_3)c_2 +(a_2 - a_4)c_1c_2)
\end{array}\right.
\]
So to get $\ker \delta_3$ we need $a_1=a_2$, $a_3=a_4$ and $a_1=a_3$, $a_2=a_4$ or $a_1=a_2=a_3=a_4$  so
\begin{flalign*}
\ker \delta_3 & = \langle s_1 s_2 s_1 a_1 (1+c_1+c_2+c_1 c_2) \rangle & \\
& = \langle s_1 s_2 s_1 a_1 (c_1+1)(c_2+1) \rangle &
\end{flalign*}
In the following calculations we will write $\B(s_j s_k)$ and $\B(s_j)$ as $s_j s_k$ and $s_j$, respectively.
\begin{flalign*}
\Ima \delta_1 & = \langle c_1-1, c_2-1 , c_1c_2 -c_2,c_1c_2 - c_1 \rangle & \\
& = \langle c_1-1, c_2-1 , c_1c_2 -1 \rangle &
\end{flalign*}
Now, we can calculate the homology group,
\[
H_0=\frac{\ker \delta_0}{\Ima \delta_1} = \frac{\langle 1,c_1,c_2,c_1 c_2 \rangle}{\Ima \delta_1} = \langle 1 \rangle \simeq \Z
\]
\begin{flalign*}
\ker \delta_1 & = \langle s_1(c_1+1), s_1 c_2(c_1+1), s_2 (c_2+1), s_2 c_1(c_2+1), & \\
& \qquad \qquad \, s_1(c_2-1) - s_2(c_1-1) \rangle & \\
& = \langle s_1(c_1+1), s_1 c_2(c_1+1), s_2 (c_2+1), s_2 c_1(c_2+1), & \\
& \qquad \qquad \, s_1(c_2-1) - s_2(c_1-1) +s_2 c_1(c_2+1)\rangle & \\
& = \langle s_1(c_1+1), s_1 c_2(c_1+1), s_2 (c_2+1), s_2 c_1(c_2+1), & \\
& \qquad \qquad \, s_1(c_2-1) + s_2(c_1 c_2+1) \rangle & \\
& = \langle s_1 (c_2+1)(c_1+1),s_1 c_2 (c_1+1), s_2 (c_1+1)(c_2+1), s_2 c_1(c_2+1), & \\
& \qquad \qquad \, s_1(c_2-1) + s_2(c_1 c_2+1) \rangle & \\
& = \langle s_1 (c_2+1)(c_1+1),s_1 c_2 (c_1+1),s_2 (c_1+1)(c_2+1),  & \\
& \qquad \qquad \, s_1 c_2(c_1+1) + s_2 c_1(c_2+1), s_1(c_2-1) + s_2(c_1 c_2+1) \rangle &
\end{flalign*}
Denoting
\[
a = \delta_2 (s_1 s_2) = s_1 (1 - c_2) - s_2 (1 + c_1 c_2)
\]
and
\[
b = \delta_2 (s_2 s_1) = s_2 (1 - c_1) - s_1 (1 + c_1 c_2)
\]
since $a(1+c_2)(1-c_1 c_2) = 0$, $b(1+c_1)(1-c_1 c_2) = 0$ and $a c_1+b c_2 = a+b=-(s_1 c_2(c_1+1) + s_2c_1(c_2+1))$, it follows that
\begin{flalign*}
\Ima \delta_2 & = \langle a, a c_1, a c_2, a c_1 c_2, b, b c_1, b c_2, b c_1 c_2\rangle &\\
& = \langle a, a c_1, a c_2, b, b c_1, b c_2 \rangle &\\
& = \langle a, a c_1, a c_2, b, b c_1 \rangle & \\
& = \langle a, a c_1, a(1+c_2), b, b(1+c_1) \rangle & \\
& = \langle a, a c_1, a(1+c_2), a+b, b(1+c_1) \rangle & \\
& = \langle a, a c_1, s_2 (1 + c_1 c_2)(1+c_2), a+b, s_1 (1 + c_1 c_2)(1+c_1) \rangle & \\
& = \langle a, a c_1, s_2 (1 + c_1)(1+c_2), a+b, s_1 (1 + c_2)(1+c_1) \rangle &
\end{flalign*}
To the 2nd element add the 1st, 3rd, -5th so that
\begin{flalign*}
& \Ima \delta_2 = \langle a, -2s_1c_2(1+c_1), s_2 (1 + c_1)(1+c_2), a+b, s_1 (1 + c_2)(1+c_1) \rangle &
\end{flalign*}
Comparing $\Ima \delta_2$ and $\ker \delta_1$, note that they have 4 equal terms and that the 5th term is double the other, since its easy to check that they are linearly independent. Then, it follows that
\[
H_1=\frac{\ker \delta_1}{\Ima \delta_2} \simeq \Z / 2\Z
\]
\begin{flalign*}
\ker \delta_2 & = \langle s_1 s_2 (c_1-1)(c_2 + 1), s_1 s_2 (c_1-1) + s_2 s_1(c_2-1), & \\
& \qquad \qquad  \, s_1 s_2 (c_1-1) + s_2 s_1 c_1(1-c_2) \rangle & \\
& = \langle s_1 s_2 (c_1-1)(c_2 + 1), s_1 s_2 (c_1-1) + s_2 s_1(c_2-1), & \\
& \qquad \qquad  \, s_1 s_2 (c_1-1) + s_2 s_1 c_1(1-c_2)-(s_1 s_2 (c_1-1) + s_2 s_1(c_2-1)) \rangle & \\
& = \langle s_1 s_2 (c_1-1)(c_2 + 1), s_1 s_2 (c_1-1) + s_2 s_1(c_2-1), & \\
& \qquad \qquad  \, s_2 s_1 (c_1+1)(1-c_2) \rangle &
\end{flalign*}
Denoting
\[
c = \delta_2 (s_1 s_2 s_1) = s_1 s_2 (c_1 - 1) + s_2 s_1 (c_2 - 1
\]
since $\rho(c_1+1)(c_2+1) = 0$, it follows that
\begin{flalign*}
\Ima \delta_3 & = \langle c, c c_1, c c_2,c c_1 c_2 \rangle & \\
& = \langle c, c c_1, c c_2 \rangle & \\
& = \langle c, \rho (c_1+1), \rho (c_2+1) \rangle & \\
& = \langle s_1 s_2 (c_1 - 1) + s_2 s_1 (c_2 - 1),s_2 s_1 (c_2 - 1)(c_1+1), s_1 s_2 (c_1 - 1)(c_2 + 1) \rangle & \\
\end{flalign*}
Note then that $\Ima \delta_3 = \ker \delta_2$ so that
\[
H_2=\frac{\ker \delta_2}{\Ima \delta_3} = 0
\]
\begin{flalign*}
& \ker \delta_3 = \langle s_1 s_2 s_3(c_1+1)(c_2+1)\rangle &
\end{flalign*}
\[
H_3=\frac{\ker \delta_3}{\Ima \delta_4} = \ker \delta_3 \simeq \Z
\]

\end{document}